\documentclass{amsart}

\usepackage{latexsym,amssymb,amsmath}
\usepackage{eufrak}
\usepackage{graphics,color,fullpage}

\usepackage{textcomp}

\usepackage{epsfig}
\usepackage[applemac]{inputenc}
\usepackage{subfigure}

\newtheorem{theorem}{Theorem}
\newtheorem{proposition}{Proposition}

\usepackage{graphicx}
\title{$3D$  polyominoes inscribed in a rectangular prism}

\author{Alain Goupil$^{1}$ and Hugo Cloutier } 
\address{ D\'{e}partement de math\'{e}matiques et d'informatique, Universit\'{e} du Qu\'{e}bec \`{a} Trois-Rivi\`{e}res, 
3351 boul des Forges, c.p. 500,
Trois-Rivi\`{e}res (QC) Canada} 
\email{alain.goupil@uqtr.ca}
\email{hugo854@yahoo.ca}

\begin{document}

\begin{abstract}
We introduce a family of $3D$ combinatorial objects that we define as  minimal $3D$ polyominoes inscribed in a rectanglar prism. These objects  are connected sets of unitary cubic cells  inscribed in a given rectangular prism and of minimal volume under this condition. They  extend the concept of $2D$ polyominoes inscribed in a rectangle defined in a previous work. Using their geometric structure and elementary combinatorial arguments, we construct generating functions of minimal $3D$ polyominoes in the form of rational functions. We also obtain a number of exact formulas and recurrences for sub-families of these polyominoes. 
\end{abstract}

\maketitle

\section{Introduction} 
Since the rise of modern combinatorics in the early 1960's, most combinatorial objects are visualized and investigated with pencil and paper  and therefore, are $2$-dimensional. A number of  extensions from  $2D$ combinatorial objects to $3D$ objects were introduced:  Ferrers diagrams were extended to plane partitions, permutations were extended to maps on a surface and to braids, $2D$ fractals were extended to $3D$ fractals and a short list of exact results for the enumeration of $3D$ objects have been produced so far (see \cite{BG},\cite{Ma}). Behind these efforts lay a fundamental question: Is $3D$ combinatorics similar to $2D$ combinatorics in the sense that it is a natural extension of notions and concepts already known in $2D$ or does it introduce new material and concepts unknown in $2D$ combinatorics ? This question was part of our motivation to begin a study of $3D$ polyominoes.  

A $2D$-polyomino is a $4$-connected set of unit square cells in the discrete plane. That is, the cells are connected by their edges.  A polyomino is inscribed in a $b\times k$ rectangle when it is contained in this rectangle and  touches each of its four sides.  Inscribed $2D$ polyominoes  with minimal area were introduced in a previous work (see \cite{GCN}) where an elementary geometric characterization  was given that permitted their enumeration and the construction of their generating functions. The geometry of an inscribed  minimal $2D$ polyomino can be described in simple terms as a {\it hook-stair-hook} structure where a hook is  formed with two mutually perpendicular rows of cells  starting on an edge of the rectangle and meeting at their corner end  (see fig. \ref{fig:3:a} red cells and \cite{GCN} for more details). A $2D$  {\it stair} is a path of connected cells beginning on one  corner of a rectangle, say north-west, and moving along the corresponding diagonal in the east - south direction (see fig. \ref{fig:3:a}, black cells and their circumscribed rectangle) to end in the opposite corner of the rectangle.

\begin{figure}
\centering
\subfigure[Generic  $2D$ minimal polyomino ] 
{ \label{fig:1:a}
    \includegraphics[width=3.0 cm]{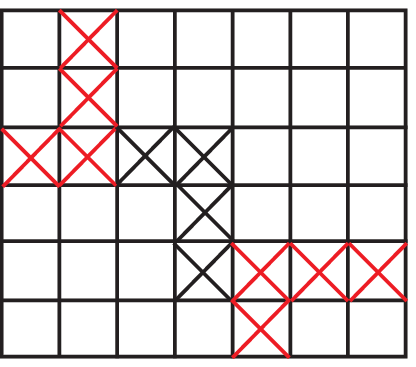}}
\hspace{.3 cm}
\subfigure[$3D$ diagonal polyomino] 
{ \label{fig:1:b}
    \includegraphics[width=6 cm]{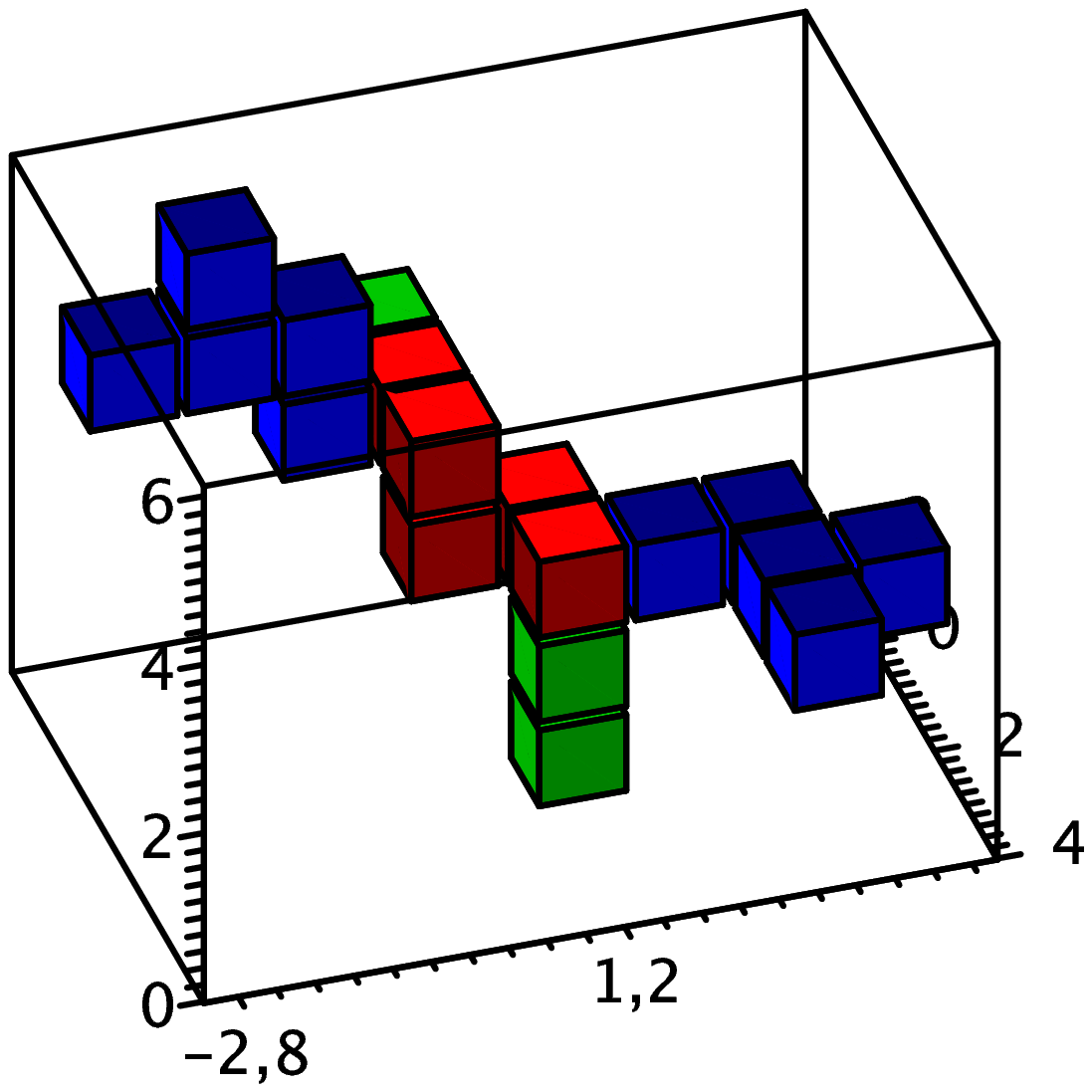}}
    \hspace{1cm}
\caption{$2D$ and $3D$ inscribed minimal polyominoes}
\label{fig1} 
\end{figure}

In this paper we introduce $3D$ polyominoes inscribed in a rectangular $b\times k\times h$ prism and we define them as collections of unit $6$-connected  cubic cells contained in the prism and touching each of its six faces. We give a geometric description of a complete collection of  families of  inscribed $3D$ polyominoes with minimal volume. This allows us to present generating functions, recurrences and exact formulas for families of minimal $3D$ inscribed polyominoes. 

$3D$ polyominoes, sometimes called polycubes, are known in the litterature and in recreational mathematics  in the context of packing problems (see \cite{Bo}) and their enumeration according to their volume is given up to volume  $16$ in \cite{Go1} as the result of a computer program. However combinatorial enumeration of inscribed $3D$ polyominoes does not seem to have been considered so far.

We will introduce three disjoint families of minimal $3D$ inscribed polyominoes and show that their union forms the complete set of  $3D$ inscribed polyominoes with minimal volume. These three families will be called respectively {\it $3D$ diagonal poliominoes}, $2D\times 2D$ polyominoes and {\it skew cross}  polyominoes.

We will use the orthogonal projection of inscribed $3D$ polyominoes on the upper face of the prism in view of the fact that an inscribed $3D$  polyomino is of minimal volume if and only if its orthogonal  projection on each rectangular face of the circumscribed prism is a $2D$ polyomino of minimal area. This is easily proved by contradiction for if a $3D$ inscribed polyomino is not minimal, then one of its projections is not $2D$ minimal. Similarly, if one projection is not minimal, then the $3D$ polyomino cannot be minimal. 

The reader will observe that pictures illustrating $3D$ polyominoes frequently appear in our  proofs. This  visual support was used in our investigation and it helped us develop $3D$ visualisation. Their absence would increase the difficulty of understanding our arguments. 

\paragraph{\bf Notations}We will use capital letters for sets and  generating functions and their corresponding lower case letters will be used for set cardinalities. For example $P_{3D,min}(b,k,h)$ will denote the set of $3D$ polyminoes inscribed in a $b\times k\times h$ rectangular prism with minimal volume, $p_{3D,min}(b,k,h)$ will be their number and $P_{3D,min}(x,y,z)=\sum_{b,k,h}p_{3D,min}(b,k,h)x^by^kz^h$  will be their generating function. We will use the convention that the  edge of length $b$ of the cube  is along the $x$ axis and similarly the lengths $k,h$ are along the $y$ and $z$ axis respectively. 

The {\it degree} of a $3D$ cell $c$ in a polyomino, denoted $deg(c)$, is the number of cells having a face in contact  with $c$ and the degree of a $2D$ cell $c$ is the number of cells with an edge contact  with $c$. All polyominoes considered in this paper are $2D$ or $3D$, always  inscribed in a rectangle or a rectangular prism and of minimal area or volume. Therefore we will often omit to specify these caractersitics of polyominoes. We will use trinomial coefficients in their standard notation $\binom{a+b+c}{a,b,c}$.
We refer the reader to [\cite{GCN}] for results and definitions on $2D$ polyominoes. 

The paper is organized as follow. In section \ref{sec1}, we  introduce diagonal $3D$ polyominoes and the subfamilies needed for their geometric description.  We give generating functions, recurrences and exact formulas for a number of these subfamilies.  In section \ref{sec2}, we define two families of non diagonal polyominoes: $2d\times 2d$ polyominoes and skew cross polyominoes with other subfamilies necessary to their description.  We give their generating functions and some exact formulas. 
In section  \ref{sec4}, we prove the main result of the paper which states that these three families of polyominoes form a complete set of $3D$ minimal inscribed polyominoes.  This result comes with  a rational form for the generating function $P_{3D,min}(x,y,z)$ of minimal $3D$ inscribed polyominoes. 


\section{Diagonal polyominoes}\label{sec1}
 In similarity with $2D$ stairs, we define a {\it $3D$ stair} as an inscribed  polyomino of minimal volume  forming a path starting in a given corner of the prism, say the north-west-back corner, and moving with unit steps in the south, east or forward direction until it reaches the opposite $3D$ diagonal corner as in figure \ref{fig:2:d}.   In what follows, we will use $3D$ stairs as components of polyominoes. 

Recall that a $2D$ corner-polyomino is a $2D$ minimal polyomino inscribed in a rectangle with a cell in a given corner of the rectangle. The number $ P_{c}(b,k)$ of $2D$ corner-polyominoes inscribed in a $b\times k$ rectangle satisfies the following recurrence and exact formula:

\begin{align}
\label{eq1}
 P_{c}(b,k)&= 1+ P_{c}(b,k-1)+ P_{c}(b-1,k)
\\
\label{eq2}
&= 2\binom{b+k-2}{b-1}-1 
\end{align}
with initial conditions $ P_{c}(b,1)= P_{c}(1,k)=1$.
Its generating function has the rational form
\begin{equation}\label{eq3}
 P_{c}(x,y)=\sum_{b,k\geq 1} P_{c}(b,k)x^by^k=\frac{2xy}{(1-x-y)}-\frac{xy}{(1-x)(1-y)}
\end{equation}
Recall also (see \cite{GCN}) that the total number of polyominoes of minimal area inscribed in a rectangle $p_{2D,min}(b,k)$ of size $b\times k$ is given by the formula
\begin{align*}
p_{2D,min}(b,k)=8\binom{b+k-2}{b-1}+2(b+k)-3bk-8
\end{align*}

We will first define and investiguate  $3D$ {\it corner-polyominoes}. A 
$3D$ {\it corner-polyomino} is  a minimal polyomino  inscribed in a prism with one cell in a given corner of the prism, say the north-west-back corner. Let $P_c(b,k,h)$ be the set of corner-polyominoes inscribed in a $b\times k\times h$ prism. 

\begin{theorem}\label{th1} For all positive  integers $b,k,h$,
the number $p_c(b,k,h)$ of $3D$ polyominoes inscribed in a prism of size $b\times k\times h$ with minimal volume and one cell in a given corner of the prism satisfies the following recurrence :

\begin{equation*}
p_c(b,k,h)=\begin{cases}
2\binom{b+k+h-3}{b-1,k-1,h-1}-1 & \text{if b=1 or k=1 or h=1}\\
1+2\binom{b+k-2}{b-1}+2\binom{b+h-2}{b-1}+2\binom{k+h-2}{k-1}-6\\
+ \;p_c(b-1,k,h)+p_c(b,k-1,h)+p_c(b,k,h-1)
& \text{otherwise}
\end{cases}
\end{equation*}
\end{theorem}
\begin{proof}  The first case is the $2D$ case. It provides the initial conditions for the $3D$ case and is obtained from equations (\ref{eq1}) and (\ref{eq2}). In the second case, observe that a corner cell has degree  one, two or three. There is exactly one $3D$ corner-polyomino of degree three inscribed in a $b\times k\times h$ prism and we call this polyomino a {\it tripod}.  This explains the term $1$ in the recurrence. When the corner cell $c$ is of degree two, then $c$ is the corner cell  of a $2D$ corner-polyomino different from a $2D$ hook that is inscribed in a face of the prism and attached to a perpendicular row of cells along an edge of the prism. A  row of cells connecting the polyomino to a face of the prism will often be considered and we will call these components {\it pilars}.  Figure  \ref{fig:2:b} illustrates this situation: the corner cell of degree two is the red cell, the $2D$ corner-polyomino is made of the red and blue cells and the set of green cells forms a pilar. The next four terms in the recurrence are thus deduced from  equation (\ref{eq2})

Now if the corner $c$ has degree one, as in figure \ref{fig:2:c}, then the polyomino starts with a $3D$ stair giving the last three terms of the recurrence and the proof is complete. Observe that the separation according to the degree of the corner cell also gives the following equivalent formulation for the  recurrence: 

\begin{align*}
P_c(b,k,h)&= tripod + ( 2D\mbox{-} corner  - 2D\mbox{-} hook) + deg1
\\
&=1+(P_c(b,k,1)+P_c(b,1,h)+P_c(1,k,h)-3)+\\
\nonumber
& (P_c(b,k,h-1)+P_c(b,k-1,h)+P_c(b-1,k,h)
\end{align*}

\begin{figure}\small
\centering
\subfigure[deg3: tripod ] 
{ \label{fig:2:a}
    \includegraphics[width=3 cm]{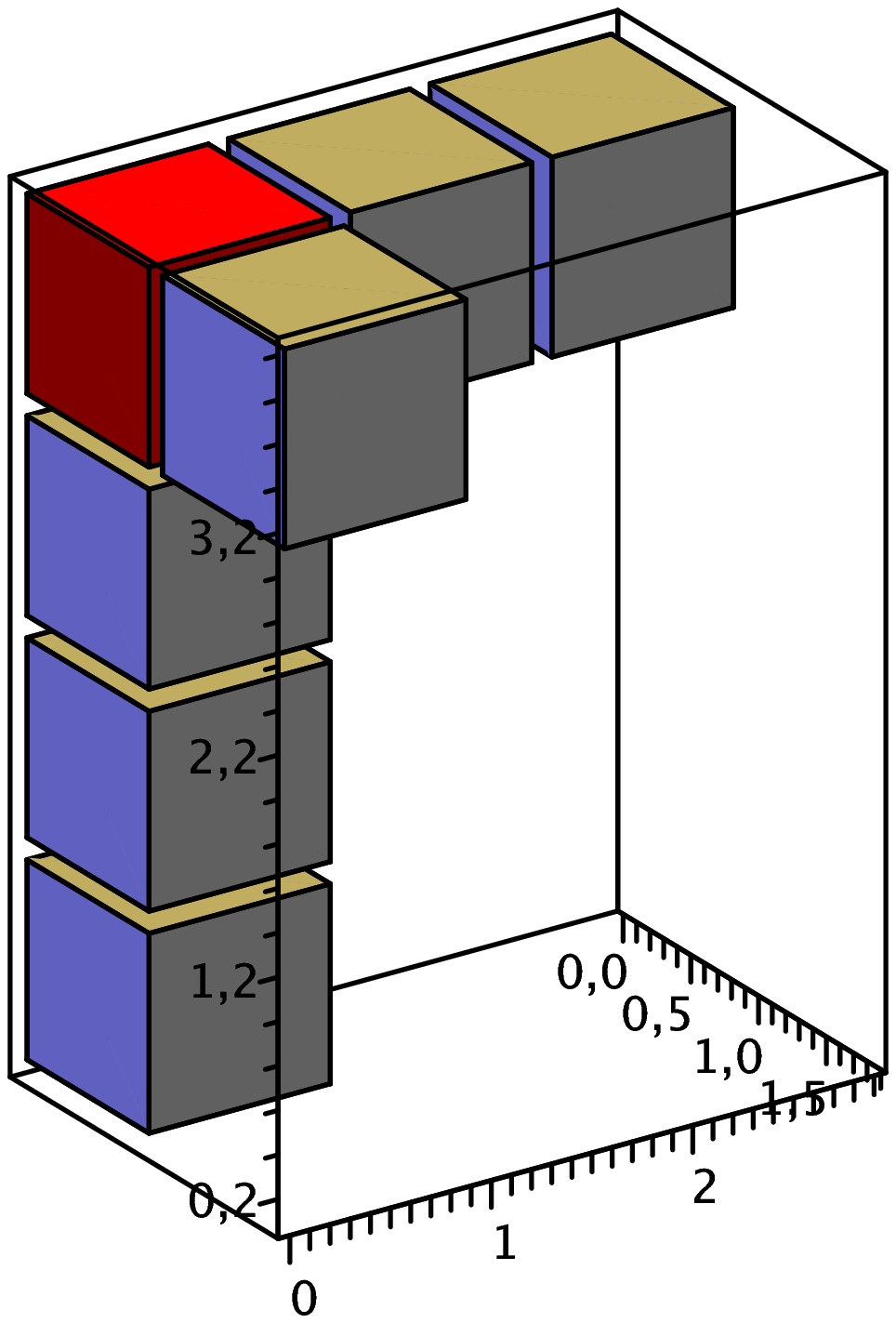}}
\hspace{.1 cm}
\subfigure[Degree two: $3D$ hook] 
{ \label{fig:2:b}
    \includegraphics[width=3 cm]{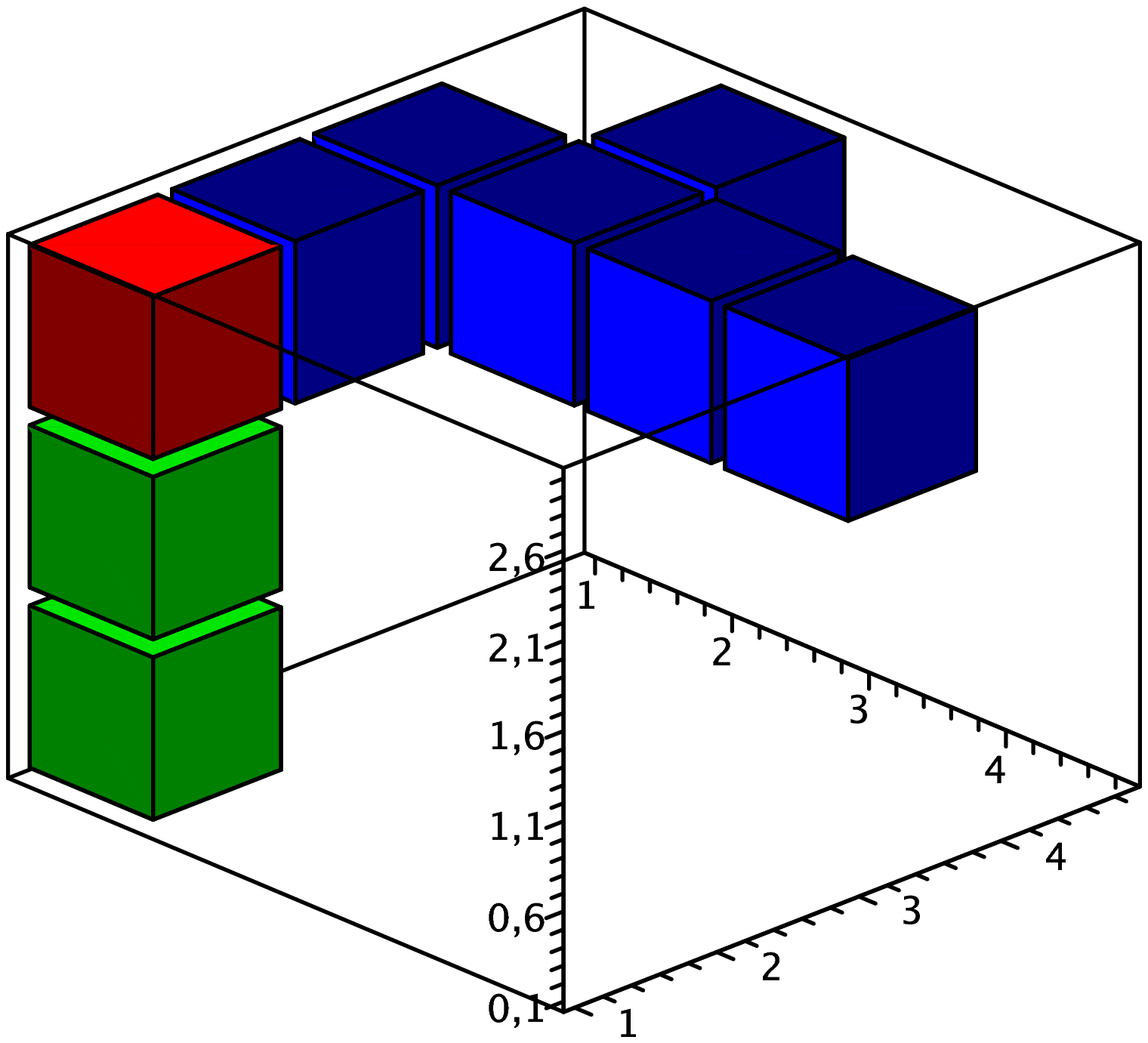}}
\hspace{.2 cm}
\subfigure[Degree one] 
{\label{fig:2:c}
    \includegraphics[width=3 cm]{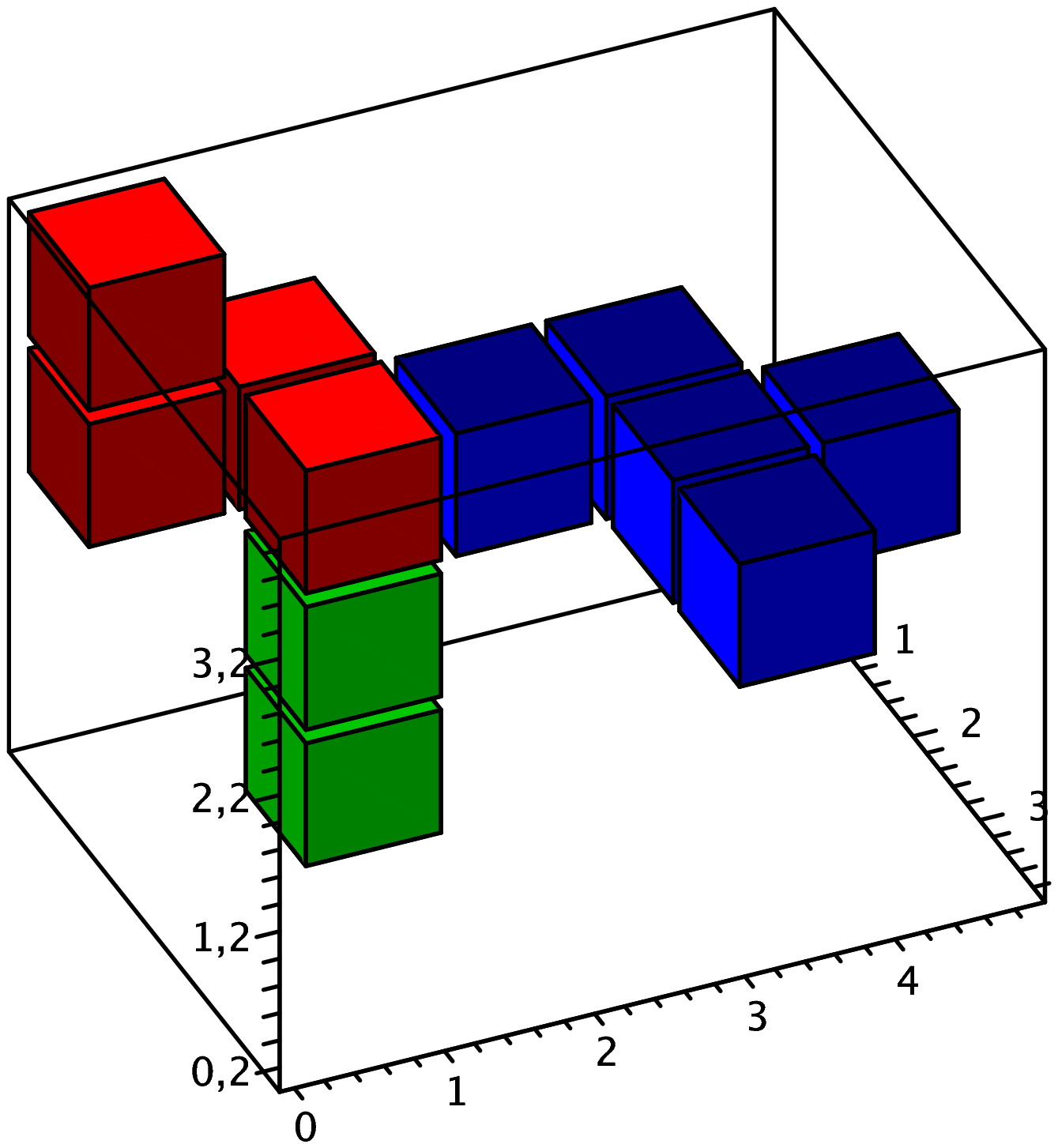}}
    \subfigure[3D stair]   
{\label{fig:2:d}
    \includegraphics[width=2.9 cm]{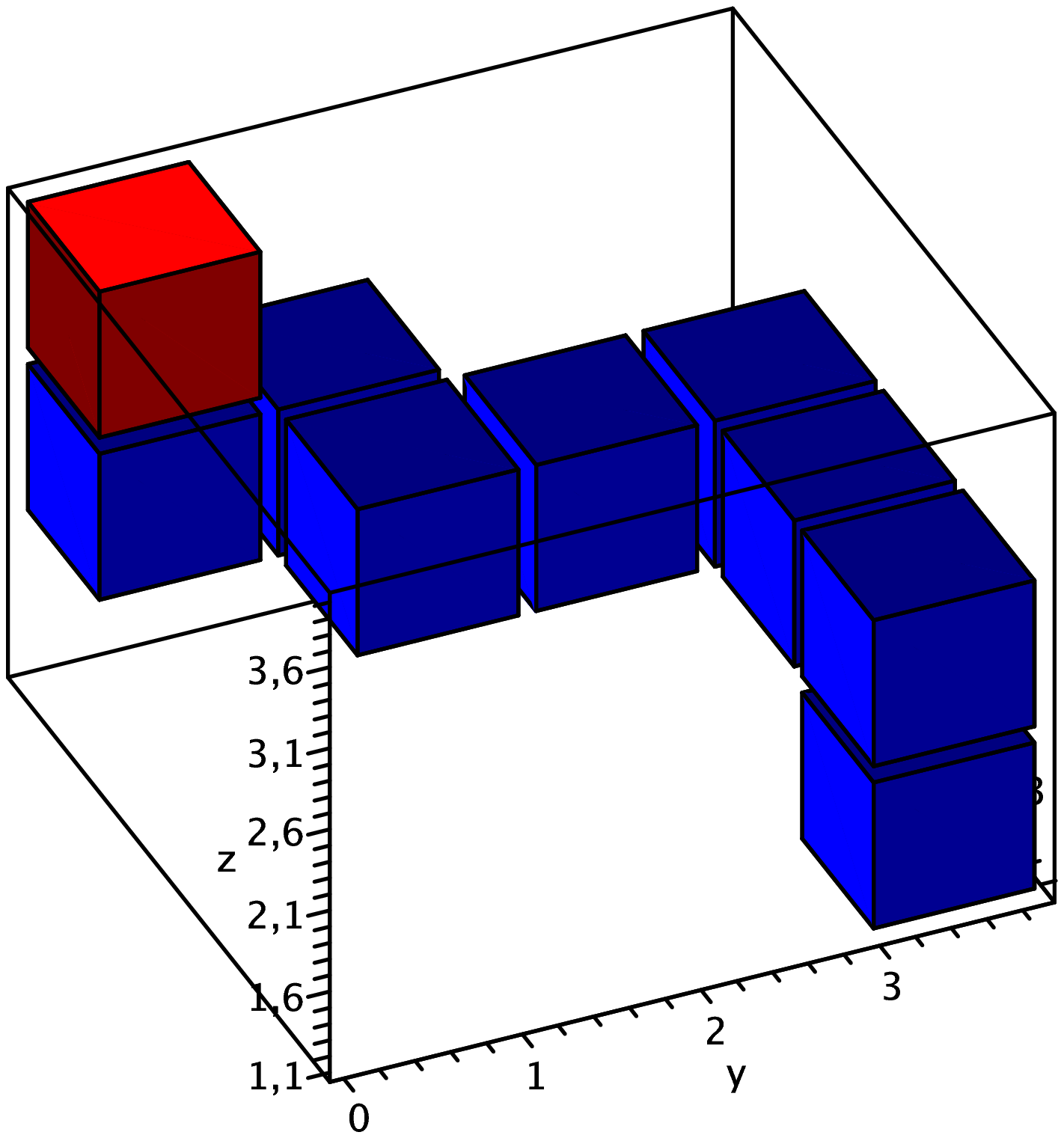}}
\caption{Corner polyominoes}
\label{fig2} 
\end{figure}
\end{proof}

\paragraph{\bf Generating function} To establish the generating function for the set of $3D$ corner-polyominoes, we will first give the generating functions $Stair(x,y,z)$, $Tripod(x,y,z)$,  $2dhook(x,y,z)$ and $Deg2(x,y,z)$ which are  $3D$ stairs, tripods,  $2D$ hooks and $3D$ corner-polyominoes of degree two respectively:
\begin{align}
\nonumber
Tripod(x,y,z)&=\sum_{i,j,k\geq 2}x^iy^jz^k=\frac{x^2y^2z^2}{(1-x)(1-y)(1-z)}\\
\nonumber
Stair(x,y,z)&= \sum_{i,j,k\geq 1}\binom{i+j+k-3}{i-1,j-1,k-1}x^iy^jz^k
=xyz\sum_{n\geq 0}(x+y+z)^n\\
\nonumber
&=\frac{xyz}{(1-x-y-z)}\\
\nonumber
2Dhook(x,y,z)&=\frac{x^2y^2z}{(1-x)(1-y)}+\frac{x^2yz^2}{(1-x)(1-z)}+\frac{xy^2z^2}{(1-z)(1-y)}\\
\label{eq7}
Deg2(x,y,z)&= 
\left[ \frac{2yz}{(1-y-z)}-\frac{2yz}{(1-y)(1-z)}  \right]\frac{x^2}{(1-x)} +\\
\nonumber
\left[ \frac{2xz}{(1-x-z)}\right.&\left.-\frac{2xz}{(1-x)(1-z)}  \right]\frac{y^2}{(1-y)} +
\left[ \frac{2xy}{(1-x-y)}-\frac{2xy}{(1-x)(1-y)}  \right]\frac{z^2}{(1-z)}
\end{align}

The proof for the rational form of these generating functions is straightforward once we understand the geometric nature of the corresponding objects: there is one tripod per prism because, by definition, their corner cell is in a given corner of the prism. The number of stairs from one corner to its diagonal opposite corner in a prism of size $b\times k\times h$ is equal to the trinomial coefficient $\binom{b+k+h-3}{b-1,k-1,h-1}$. $2D$-hooks appear on a slice parallel to one of the faces so we have three terms, one for each coordinate plane. The generating function for corner-polyominoes of degree two (equation (\ref{eq7})) is  directly obtained from its definition: a $2D$ corner of degree one perpendicular to a $pilar$.

Now we are ready to use these building  blocs. For instance a $3D$ corner of degree one   always begins as a $3D$ stair of length at least two connected to a $3D$ corner of any degree. The generating function $Deg1(x,y,z)$ of  $3D$ corners of degree one is thus 

\begin{align*} 
Deg1(x,y,z)=\left( Stair(x,y,z)-xyz\right)
\left( 1+ Tripod+Deg2 +2Dhook \right)
\end{align*}

Since we now have the generating functions for corner-polyominoes of degree one, two and three, we deduce the following result.

\begin{proposition}
The generating function $P_c(x,y,z)$ for $3D$ corner-polyominoes is the following :
\begin{align}  \label{eq9}
\nonumber
P_c(x,y,z)&= \sum_{b,k,h\geq 1}p_c(b,k,h)x^by^kz^h\\
&=
Stair(x,y,z)\left[ 1+\frac{Tripod(x,y,z)+Deg2(x,y,z)+2Dhook(x,y,z)}{xyz}
\right]
\end{align}
\end{proposition}
\begin{proof}
This is an immediate consequence of the fact that a $3D$ corner-polyomino is  the connection of a $3D$ stair with a $3D$ corner-polyomino of arbitrary degree. In equation (\ref{eq9}), we decide that  the corner cell common to a $3D$ corner-polyomino and a $3D$ stair belongs to the stair so we divide the generating function of the former by $xyz$.
\end{proof}

\begin{theorem}\small
For all positive  integers $b,k,h$, we have 
\begin{align}\label{eq10}
\nonumber
p_c(b,k,h)&=4\binom{b+h-2}{h-1}\binom{b+k+h-3}{b+h-2}+
\sum_{i=0}^{h-2}(-1)^i \binom{b+h-4-2i}{h-2-i} \binom{b+k+h-4-i}{b+h-3-2i}\\
&-2\left[ \binom{b+h-2}{b-1}+\binom{b+k-2}{k-1}+\binom{k+h-2}{h-1}\right] + 3-\frac{(1+(-1)^h)}{2}
\end{align}
\end{theorem}
\begin{proof}
By induction on $b+k+h$. If $h=1$ then the prism is reduced to a rectangle in the $xy$ plane and formula (\ref{eq10}) gives $p_c(b,k,1)=2\binom{b+k-2}{b-1}-1$ which agrees with equation (\ref{eq2}). The same argument is true for $b=1$ and $k=1$. Suppose that formula (\ref{eq3}) is true for a prism of size $b\times k\times h$ with $b,k,h\geq 2$. We have 
\begin{align*}
p_c(b,k,h+1)&=
1+2\binom{b+k-2}{b-1}+2\binom{b+h-1}{b-1}+2\binom{k+h-1}{k-1}-6\\
&+p_c(b-1,k,h+1)+p_c(b,k-1,h+1)+p_c(b,k,h)
\end{align*}
by theorem \ref{th1} and by induction hypothesis we obtain
\begin{align*}
p_c(b,k,h+1)&=1+2\binom{b+k-2}{b-1}+2\binom{b+h-1}{b-1}+2\binom{k+h-1}{k-1}-6\\
&+4\binom{b+h-2}{h}\binom{b+k+h-3}{b+h-2}+
\sum_{i=0}^{h-1}(-1)^i \binom{b+h-4-2i}{h-1-i} \binom{b+k+h-4-i}{b+h-3-2i}\\
&-2\left[ \binom{b+h-2}{b-2}+\binom{b+k-3}{k-1}+\binom{k+h-1}{h}\right] + 3-\frac{(1+(-1)^{h+1})}{2}\\
&+4\binom{b+h-1}{h}\binom{b+k+h-3}{b+h-1}+
\sum_{i=0}^{h-1}(-1)^i \binom{b+h-3-2i}{h-1-i} \binom{b+k+h-4-i}{b+h-2-2i}\\
&-2\left[ \binom{b+h-1}{b-1}+\binom{b+k-3}{k-2}+\binom{k+h-2}{h}\right] + 3-\frac{(1+(-1)^{h+1})}{2}\\
&+4\binom{b+h-2}{h-1}\binom{b+k+h-3}{b+h-2}+
\sum_{i=0}^{h-2}(-1)^i \binom{b+h-4-2i}{h-2-i} \binom{b+k+h-4-i}{b+h-3-2i}\\
&-2\left[ \binom{b+h-2}{b-1}+\binom{b+k-2}{k-1}+\binom{k+h-2}{h-1}\right] + 3-\frac{(1+(-1)^h)}{2}\\
&=4\binom{b+h-1}{h-1}\binom{b+k+h-2}{b+h-1}+
\sum_{i=0}^{h-1}(-1)^i \binom{b+h-3-2i}{h-1-i} \binom{b+k+h-3-i}{b+h-2-2i}\\
&-2\left[ \binom{b+h-1}{b-1}+\binom{b+k-2}{k-1}+\binom{k+h-1}{h}\right] + 3-\frac{(1+(-1)^{h+1})}{2}
\end{align*}
which is what we wanted to prove.
\end{proof}
It is now possible to construct formulas for the set of polyominoes along one given diagonal of the prism. We define {\it diagonal polyominoes} as inscribed polyominoes of minimal volume formed with three pieces: two  hooks on each corner of a diagonal of the prism  connected by a stair in contact with their corner cell  (see figure \ref{fig:1:b}). By a hook we mean either a $3D$ corner-polyomino with corner of degree two or three or a $2D$ hook. From this definition, we deduce the rational form of the generating function of diagonal polyominoes.

\begin{proposition} The generating function $1Diag(x,y,z)$ of diagonal polyominoes along one given diagonal of a prism is the following
\begin{align}\label{eq11}
1Diag(x,y,z)=Stair(x,y,z)\left[ 1+\frac{Tripod(x,y,z)+Deg2(x,y,z)+2Dhook(x,y,z)}{xyz}
\right]^2
\end{align}
\end{proposition}
\begin{proof}
This is a direct consequence of the definition of diagonal polyominoes, tripods, stairs, corner-polyominoes and $2D$ hooks. Notice that the number $1$ inside the brackets of equation (\ref{eq11}) stands for the fact that $3D$ hooks could be absent and we divide by $xyz$ the next term because we arbitrarily decide that the cell common to a  hook and a stair belongs to the stair so that we remove it from the hook with this division. 
\end{proof}

In the next step, we want to count the total number of diagonal polyominoes in a prism. There are four $3D$ diagonals in a prism. If a polyomino belongs to exactly two diagonals, then the two diagonals always define a plane perpendicular to two parallel faces of the prism. The orthogonal projection of the polyomino on these faces must be a $2D$ minimal polyomino and therefore this projection has the generic form {\it hook-stair-hook}  of a $2D$ minimal polyomino. The projection of the two $3D$ diagonals on any other face are the two diagonal of these rectangles. Since the only $2D$ polyomino that belongs to two diagonals of a rectangle is a $2D$ cross, i.e. two perpendicular rows of cells, the projection of the polyomino on the other faces is always a $2D$ cross. This has consequences on the  form of any $3D$ polyomino along two diagonals which must be made of a {\it full pilar}, i.e. a pilar connecting  two opposite faces, connected to a perpendicular $2D$ generic polyomino inscribed in a full $2D$ slice of the prism  (see the blue part in figure \ref{fig:3:b}. Moreover the full pilar must meet the orthogonal $2D$ polyomino on its stair part. 

Now if a diagonal polyomino belongs to three diagonals, then its projection on each face of the prism is a $2D$ cross. The only polyomino whose projection  on all faces is a $2D$ cross must be a $3D$ cross which also belongs to four diagonals (see figure \ref{fig:3:a}).

\begin{figure}
\centering
\subfigure[  $3D$ cross ] 
{ \label{fig:3:a}
    \includegraphics[width=4.4 cm]{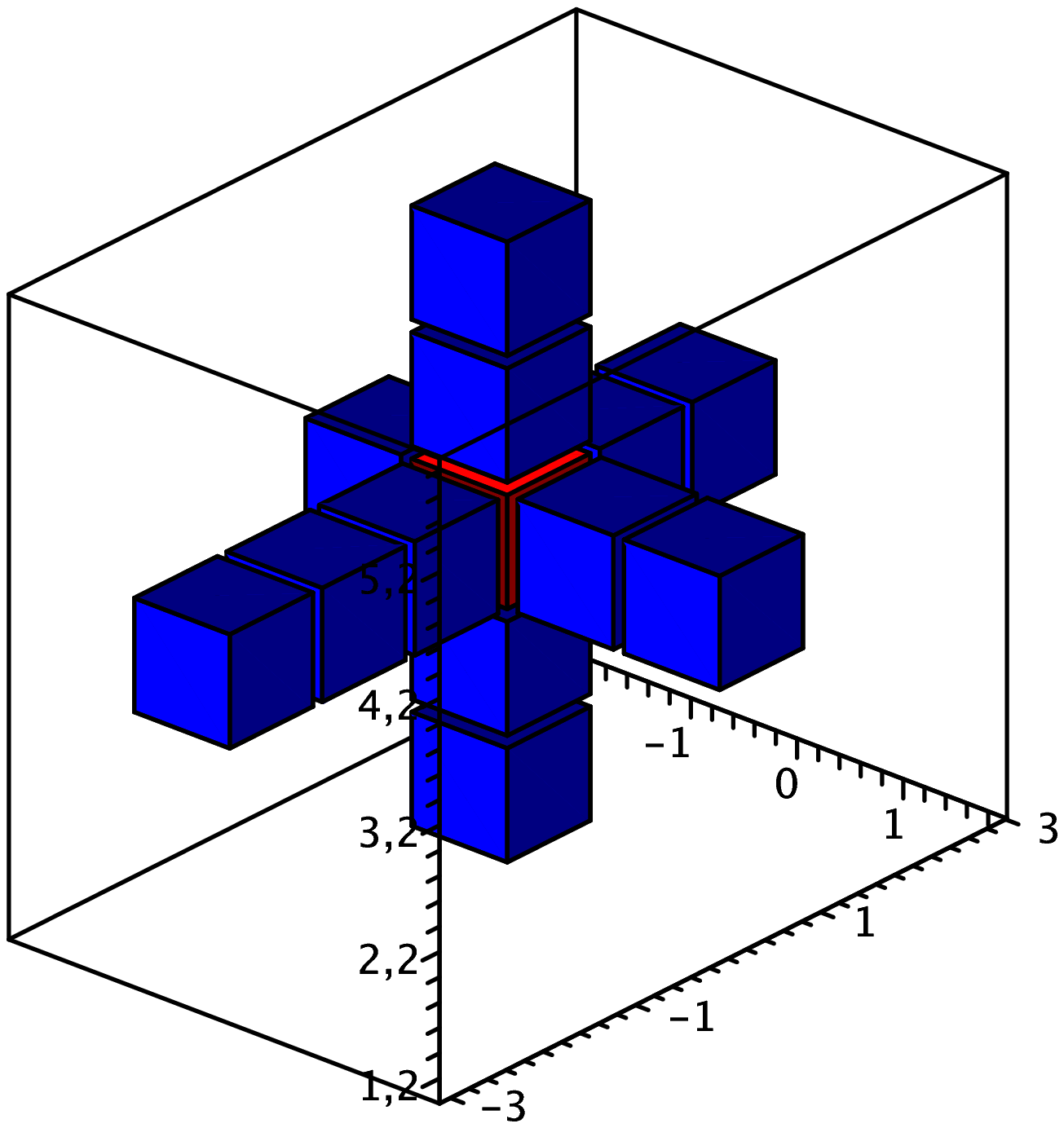}}
\hspace{.3 cm}
\subfigure[A polyomino on two diagonals] 
{ \label{fig:3:b}
    \includegraphics[width=4.6 cm]{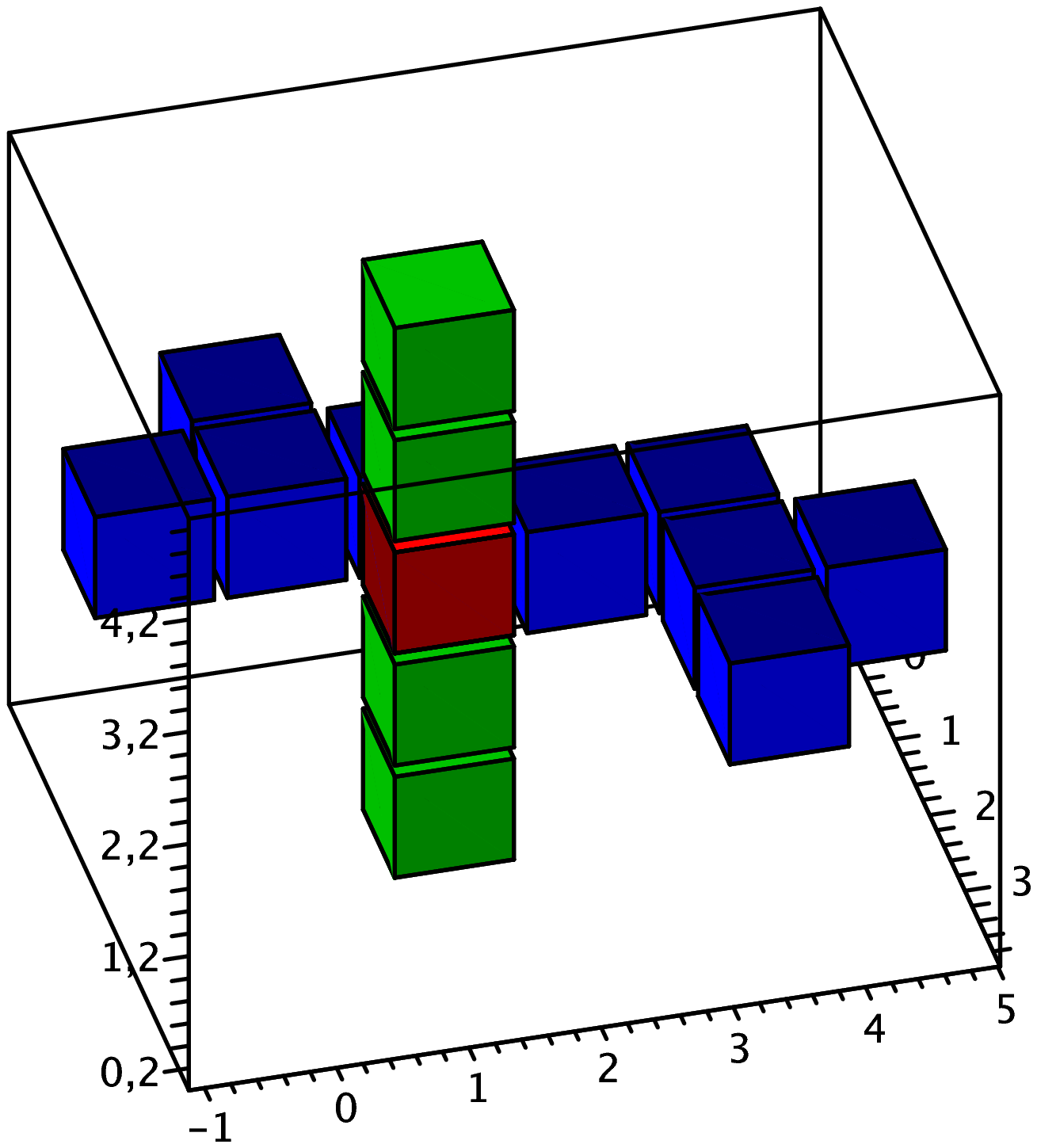}}
    \hspace{1cm}
\caption{Diagonal polyominoes on more than one diagonal}
\label{fig3} 
\end{figure}

\paragraph{\bf Polyominoes along two diagonals}
The generic form of polyominoes on two diagonals can be described as two $2D$ corner-polyominoes sharing their corner cell which belong to a full pilar perpendicular to the corner-polyominoes. Since we already know the generating function for $2D$ corner-polyominoes, it is easy to deduce  the generating function for diagonal polyominoes belonging to two and three diagonals. 

\begin{proposition}The number $2diag_z(b,k,h)$ of $3D$ diagonal polyominoes belonging to the two diagonals perpendicular to the $xy$ face of a prism such that the projection of these two diagonals has a vertex in the upper left corner of the face of size $b\times k$ has the following generating function
\begin{align*}
\nonumber
2diag_z(x,y,z)&=\sum_{b,k,h\geq 1}2diag_z(b,k,h)x^by^kz^h\\
&=\frac{1}{xy}\left(\frac{2xy}{(1-x-y)}-\frac{xy}{(1-x)(1-y)}\right)^2
\times \frac{z}{(1-z)^2}
\end{align*}
\end{proposition}
\begin{proof}
This is immediate from equation (\ref{eq3}) and the fact that these polyominoes have the geometric structure $2D \; corner \times (2D \; corner - corner \; cell)\times pilar$. \\
\end{proof}

\paragraph{\bf $3D$ crosses}Next we need the generating function $3Dcross(x,y,z)$ of $3D$ crosses which are the $3D$ minimal polyominoes made only of pilars, at least three, meeting on one common cell $c$  (see figure \ref{fig:3:a}). Observe that for a prism of size $b\times k\times h$ with $b,k,h\geq 2$, there are $bkh$ cross polyominoes inscribed in that prism and only one if  any two of these three parameters equals one.
We will only consider crosses in a box of size at least $2\times 2\times 2$.  We thus have:
\begin{align*}
\nonumber
3Dcross(x,y,z)&= \sum_{b,k,h\geq 2}bkhx^{b}y^{k}z^h={\frac {{x}^{2} \left( 2-x \right) {y}^{2} \left( 2-y \right) {z}^{2}
 \left( 2-z \right) }{ \left( 1-x \right) ^{2} \left( 1-y \right) ^{2}
 \left( 1-z \right) ^{2}}}
\end{align*}

\begin{proposition}
The generating function $Diag(x,y,z)$ of the total number of diagonal polyominoes is the following 
\begin{align}\label{eq15}
\nonumber
Diag(x,y,z)=\sum_{b,k,h\geq 2}diag&(b,k,h)x^by^kz^h \\
=4\times 1Diag(x,y,z)&
-2\times\left( 2diag_z(x,y,z)+2diag_y(x,y,z)+2diag_x(x,y,z)
\right)+3\times 3Dcross(x,y,z)
\end{align}
\end{proposition}
\begin{proof}
In order to count all $3D$ diagonal polyominoes, we use some inclusion-exclusion principle. Here are the steps:
1- Count polyominoes along one diagonal and multiply by four. 2- The polyominoes that belong to two diagonals or more were counted twice or more so for each pair of $3D$ diagonals, remove the polyominoes belonging to those two diagonals. 3- The polyominoes belonging to three diagonals, and thus to four, were counted four times in the first step, removed six times in the second step and so must be added three times to be counted once. 

Notice that this inclusion-exclusion argument is not valid for degenerate prisms that  have one side of length one and for their corresponding terms in the  generating function (\ref{eq15}).   
\end{proof}
In table \ref{tab1} below, we present the numbers $diag(n,n,n)$ of $3D$ diagonal polyominoes as well as cardinality of other sets of polyominoes that are inscribed in a cubic prism of size $n\times n\times n$  that are described in this paper.

\begin{table}[htdp]
\begin{center}
\begin{tabular}{|c|c|c|c|c|c|c|c|c|c|c|c|c|c|c|}
\hline
$n$&1&2&3&4&5&6&7&8\\
\hline
$diag(n,n,n)$&1&32& 2271& 79936& 2103269& 49998072&1163531779&27263453288\\
\hline
$P_{2D\times2D}(n,n,n)$&0&0&66&2256&34092&352992&2994750&22756896\\
\hline
$SC_a(n,n,n)$&0&0&48&3456&85008&1321344&16174416&172476672\\
\hline
$SC_b(n,n,n)$&0&0&16&1408&33776&505472&5998512&62474496\\
\hline
$P_{3D,min}(n,n,n)$&1&32&2401&87056&2256145&52177880&1188699457&27521161352\\
\hline
\end{tabular}
\end{center}
\caption{Minimal polyominoes inscribed in a $n\times n\times n$ prism} 
\label{tab1}
\end{table}

\section{Non diagonal polyominoes}
\label{sec2}
Does there exists minimal $3D$ polyominoes that are not diagonals ? The answer is yes and figure \ref{fig4}  shows a sample of these objects. For instance the polyomino in figure \ref{fig:4:a} is not diagonal because it is impossible to identify a corner-polyomino in a subprism having three faces in common with the circumscribed prism. This polyomino can be seen as the juxtaposition of two perpendicular $2D$ polyominoes each with contact cell that is not a corner cell.  This is our definition for the family of non diagonal minimal polyominoes that we call $2D\times 2D$ polyominoes.  

 
\subsection{$2D\times 2D$ polyominoes} In what follows, we establish the generating function for $2D\times 2D$ polyominoes. For that purpose, we split these polyominoes in three parts, each part corresponding to one color in figure \ref{fig:4:a}. The central part, made of green cells with red corners, will be called a {\it skew hook}. 
It consists of three mutually orthogonal segments of cells. The two end segments touch a face of the prism and so are  pilars with at least one cell. They touch the middle segment on its end  cells. These two end cells are the contact cells of the two other parts (one in blue and one in yellow in figure \ref{fig:4:a}). If we discard the two pilars, each end cell of the middle segment can be seen as the corner cell of a $2D$ corner-polyominoe.  The two $2D$ corner-polyominoes with their associated pilars  are perpendicular and each one goes from one face to its opposite face. Notice that the smallest prism that contains a $2D\times 2D$ polyominoe has size $2\times 3\times 3$ and in that case, the polyominoes are made of two perpendicular full pilars that are the discrete version of euclidian skew lines.

We begin with the generating function of skew hooks. This is quite elementary when we consider that each pilar contains at least one green cell and the central segment contains two red end cells but not necessarily green cells. In order to fix ideas, we agree that the yellow $2D$ polyomino is in the $yz$ plane with $z$ length at least two if we count the red corner cell. The blue $2D$ polyomino is in the $xy$ plane. If we decide that we do not count the contribution in $x$ and $z$ of the central segment and the contribution in $y$  of the red corner cells. we have the following generating function  $SH(x,y,z)$ for skew hooks :
\begin{equation*}
SH(x,y,z)=\frac{x}{(1-x)}\times  \frac{1}{(1-y)}\times \frac{z}{(1-z)}
\end{equation*}

The yellow $2D$ corner  polyominoes  in the $yz$ plane of  $z$ height at least $2$ are obtained from equation (\ref{eq3}):
\begin{equation*}
2D_{c,z\geq 2}(y,z)=yz\left(\frac{2}{(1-y-z)}- \frac{1}{(1-y)(1-z)}- \frac{1}{(1-y)}\right).
\end{equation*}

Then the blue $2D$ corner  polyominoes  in the $xy$ plane 
of  $x$ length at least $2$ :
\begin{equation*}
2D_{c,x\geq 2}(x,y)=xy\left(\frac{2}{(1-x-y)}- \frac{1}{(1-x)(1-y)}- \frac{1}{(1-y)}\right).
\end{equation*}

In order to assemble these three components, observe that if we fix the vertical pilar and the yellow $2D$ corner, then the horizontal green pilar may take two directions that determines the direction of the blue $2D$ polyomino which is equivalent to multiply by  two the number of  blue $2D$ corner-polyominoes and remove the $2D$ crosses which would be counted twice otherwise. We do the same for the yellow $2D$ corner-polyominoes. Finally, observe that the yellow polyomino could be on the left rather than on the right of the central part which multiplies by two again the number of polyominoes and we obtain  the following generating function for $2D\times 2D$ polyominoes with orthogonal planes $xy$ and $yz$.

\begin{equation}\label{eq18}
P_{xy\times yz}(x,y,z)=2\left(2\cdot 2D_{c,x\geq 2}(x,y)-  \frac{x^2y}{(1-x)(1-y)}\right)\cdot SH\cdot 
\left(2\cdot 2D_{c,z\geq 2}(y,z)-  \frac{yz^2}{(1-y)(1-z)}\right).
\end{equation}

\begin{figure}
\centering
\subfigure[ case 1 ] 
{ \label{fig:4:a}
    \includegraphics[width=4.4 cm]{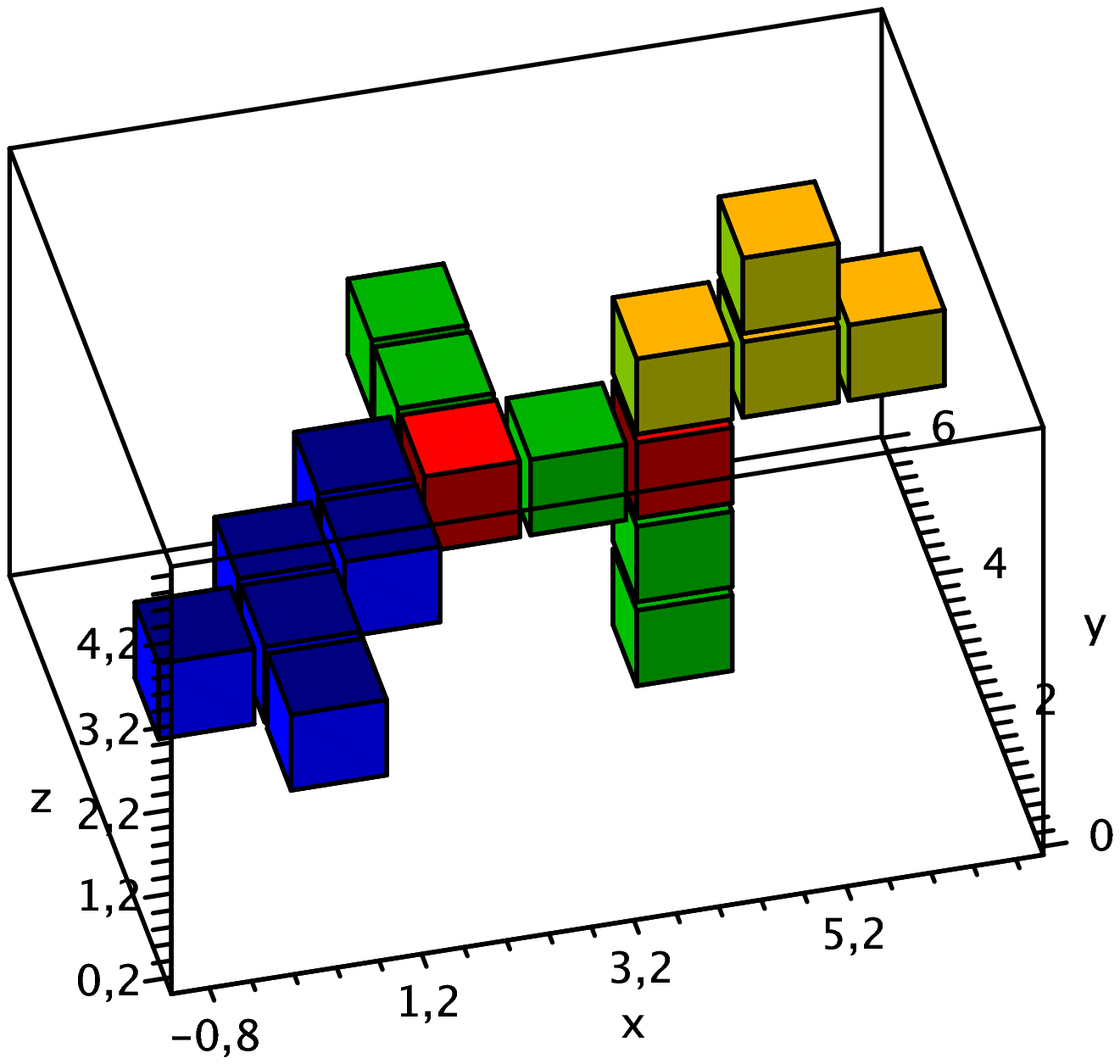}}
\hspace{.3 cm}
\subfigure[case 2] 
{ \label{fig:4:b}
    \includegraphics[width=4.3 cm]{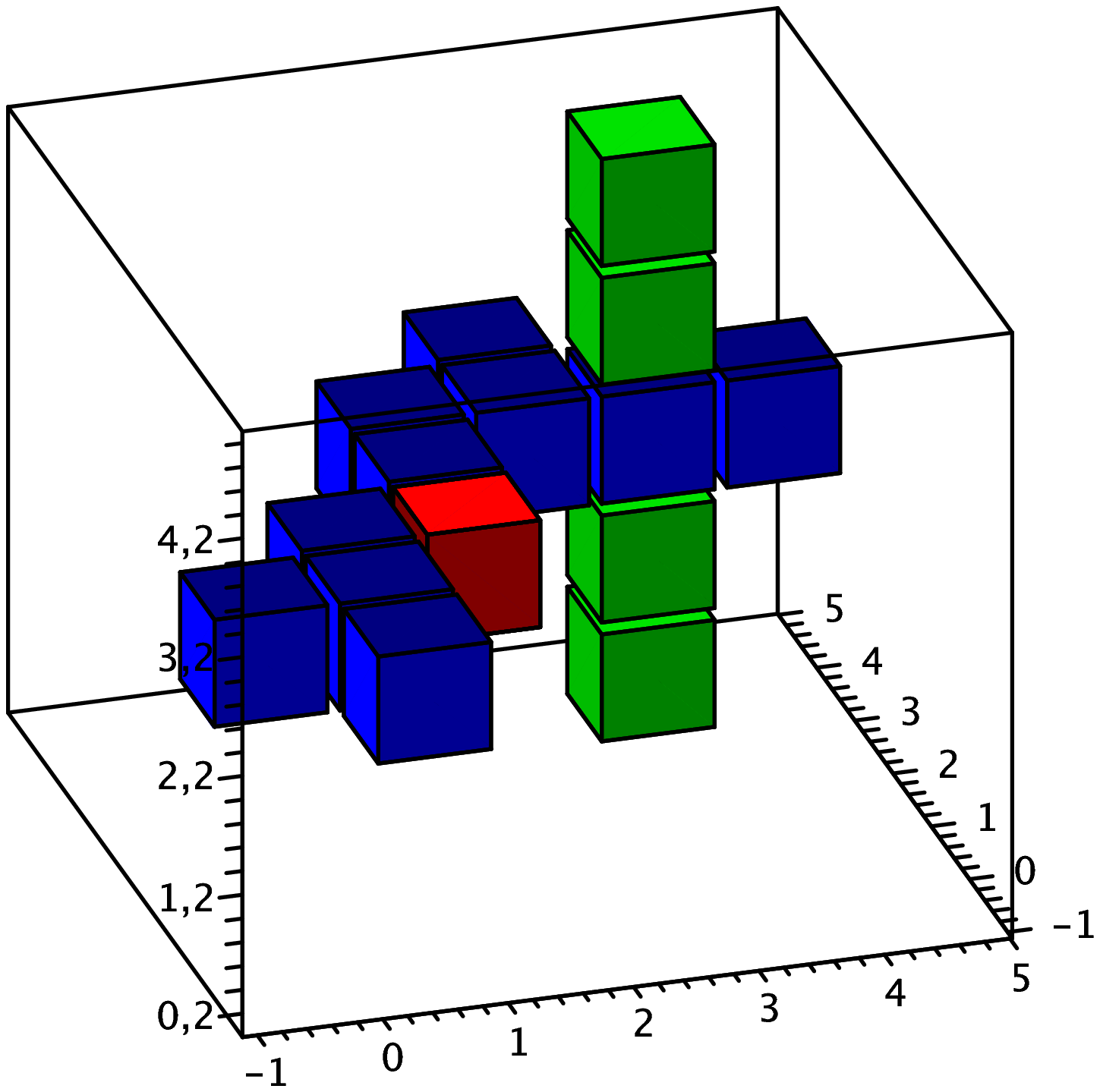}}
    \hspace{.3 cm}
    \subfigure[case 3] 
{ \label{fig:4:c}
    \includegraphics[width=4.4 cm]{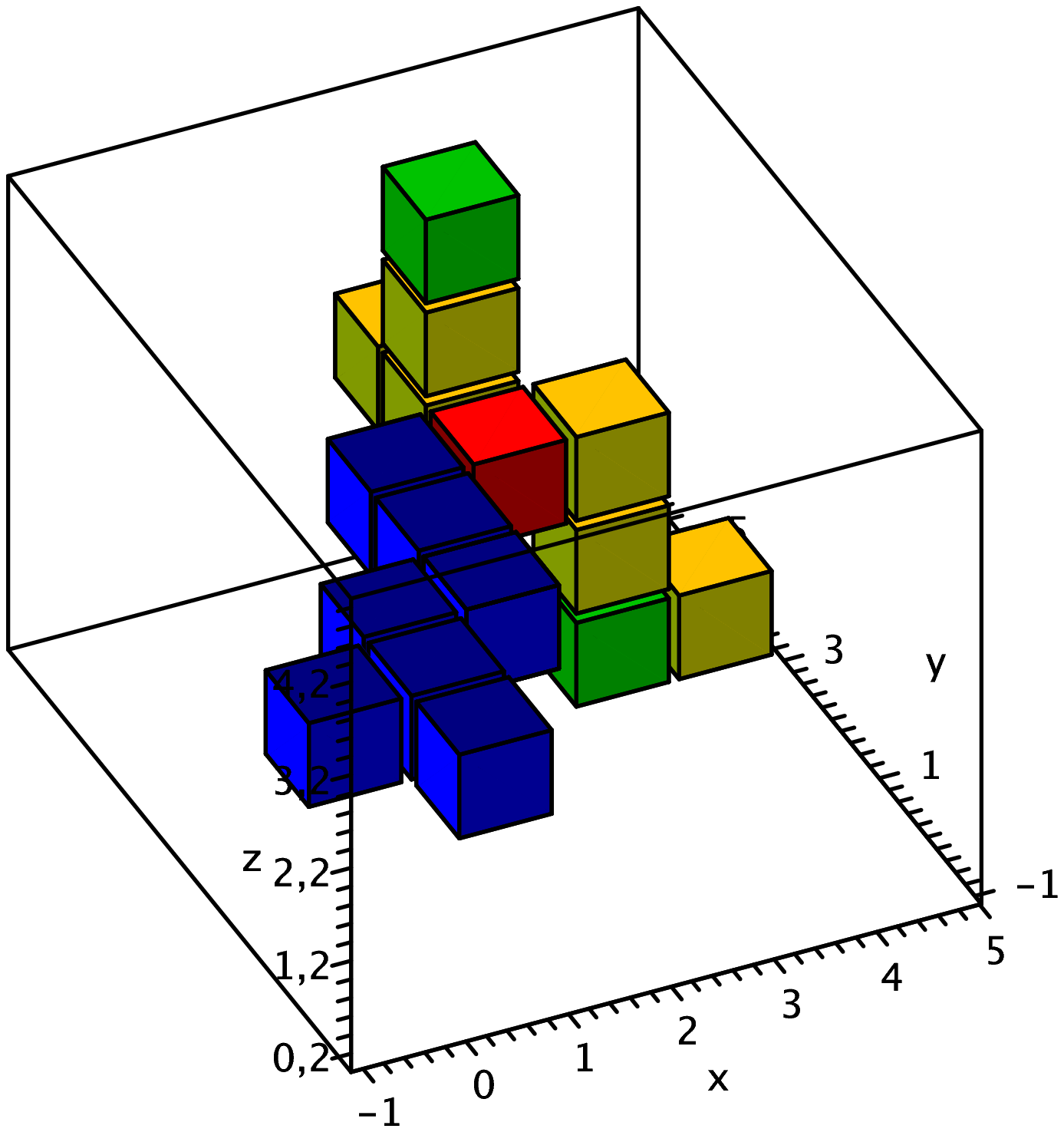}}
    \hspace{1cm}
\caption{Non diagonal Polyominoes  }
\label{fig4} 
\end{figure}

Finally, observing that two pairs of orthogonal planes determine two disjoint sets of $2D\times 2D$ polyominoes, we obtain the generating function $P_{2D\times 2D}(x,y,z)$ for the total number of non diagonal $2D\times2D$ polyominoes by adding the three generating functions corresponding to each pair of orthogonal planes:
\begin{equation}\label{eq19}
P_{2D\times 2D}(x,y,z)=P_{xy\times yz}+P_{xy\times xz}+P_{xz\times yz}.
\end{equation}

\subsection{Skew cross polyominoes} We define our second family of non diagonal polyominoes as follow. A {\it skew cross} polyomino starts with a central cell $c$ of degree three which is the corner cell of three $2D$ corner-polyominoes mutually perpendicular. We partition this family in two types. {\bf Type a)} The cell $c$ has two parallel contact faces.  {\bf Type b)} The three contact faces of the central cell $c$ are incident to a vertex of $c$. These two families are illustrated in figure \ref{fig5}.
 \begin{figure}
 \centering
\subfigure[] 
{ \label{fig:5:a}
    \includegraphics[width=2.9 cm]{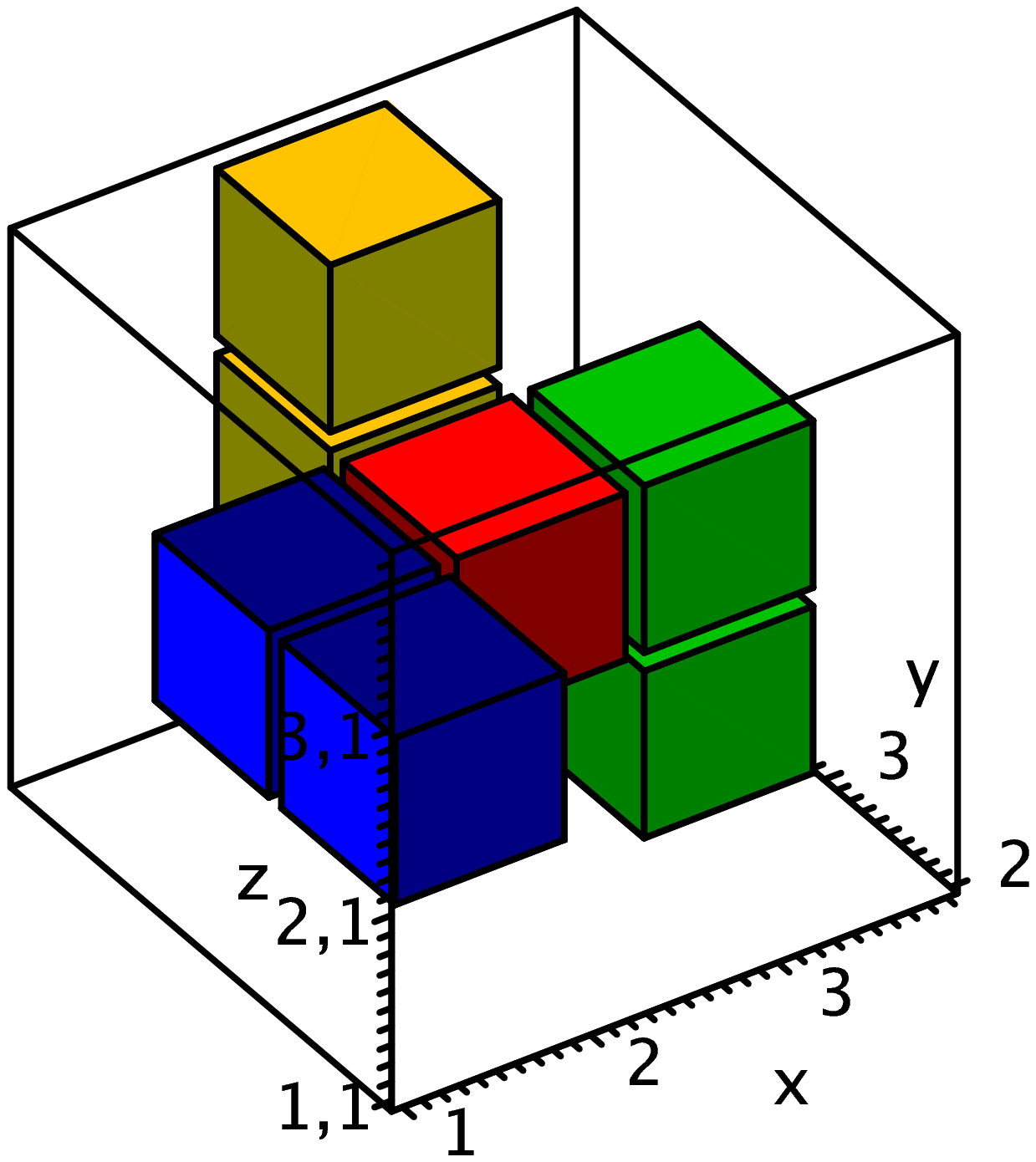}}
\hspace{.0 cm}
\subfigure[] 
{\label{fig:5:b}
    \includegraphics[width=2.9 cm]{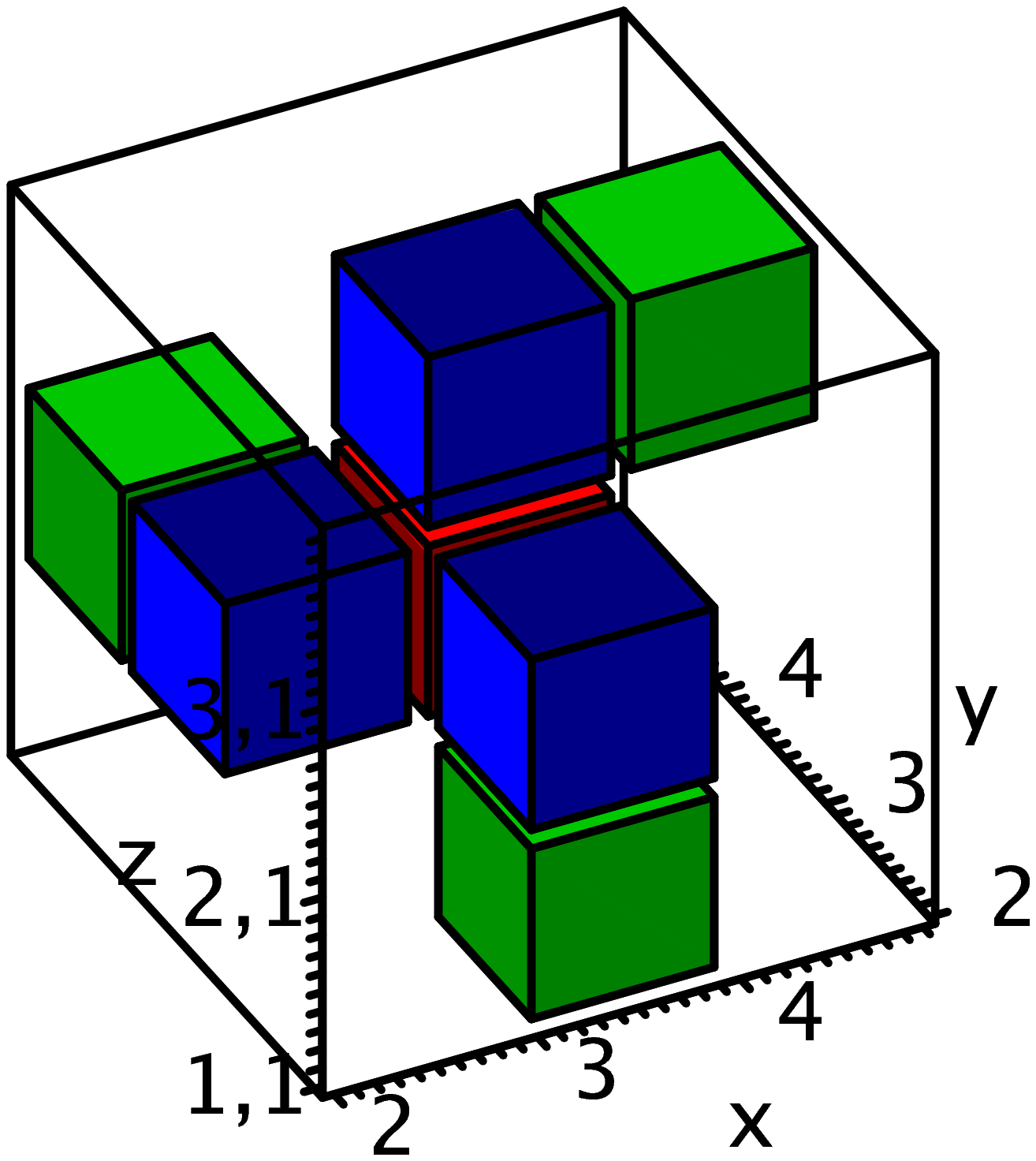}}
    \hspace{.0 cm}
    \caption{Skew cross polyominoes} \label{fig5}
    \end{figure}
    
 \paragraph{\bf Type a)} We start by establishing the generating function for each of the three $2D$ corner-polyominoes needed to obtain a skew cross polyomino of type $a)$. To fix the ideas, suppose that the three contact faces of the cell $c$ have already been chosen and that the $2D$ corner-polyomino red and green is in the $yz$ plane, the yellow part is in the $xz$ plane and the blue part is in the $xy$ plane as illustrated in figure \ref{fig:5:a}. We have the choice between the red central cell $c$ and the cell in contact with it as the corner cell of the $2D$ corner-polyomino. We choose the cell in contact with $c$. For the $2D$ corner-polyomino in the $yz$ plane, the $z$ length must be at least $2$ and the generating function is 
 \begin{equation}\label{eq20}
P_{c,z\geq 2}(y,z) =yz\left(\frac{2}{(1-y-z)}- \frac{1}{(1-y)(1-z)}- \frac{1}{(1-y)}\right)
\end{equation}
Similarly for the $2D$ corner-polyominoes in the $xy$ and $xz$ planes, we obtain
\begin{align}\label{eq21}
P_{c,x\geq 2}(x,y)=xy\left(\frac{2}{(1-x-y)}- \frac{1}{(1-x)(1-y)}- \frac{1}{(1-y)}\right)\\
\label{eq22}
P_{c,z\geq 2}(x,z)=xz\left(\frac{2}{(1-x-z)}- \frac{1}{(1-x)(1-z)}- \frac{1}{(1-x)}\right)
\end{align}
The product of the three  series (\ref{eq20}),(\ref{eq21}), (\ref{eq22}) gives the generating function of skew cross polyominoes of type $a$ with preselected faces of the central cell provided we adjust with the fact that the $z$ length of the cell $c$ was counted twice and its $y$ lenght was not counted. Now once the faces of $c$ are chosen, there is some freedom for the direction of the $2D$ corner-polyominoes. Indeed, if we choose first one of the two  directions of the corner-polyomino coming from the face between opposite faces, then we still have to choose between two directions for another $2D$ corner-polyomino. For two faces in the $xz$ plane, we have two choices for a face $yz$. Thus the generating function $SC_{a1}(x,y,z)$ of skew crosses  of type $a$ when two faces  in plane $xz$ and one face in the plane $yz$ are chosen is :
\begin{align}\label{eq24}
\nonumber
SC_{a1}(x,y,z)&= 
\frac{4y}{z}P_{c,z\geq 2}(y,z)P_{c,x\geq 2}(x,y)P_{c,z\geq 2}(x,z)\\
&=\frac{4x^3y^3z^3(1+x-z)(1+y-z)(1+y-x)}{(1-x)^2(1-y)^2(1-z)^2
(1-y-z)(1-y-x)(1-x-z)}
\end{align}
Knowing that there are 12 triplets of faces of type $a$ on a cell $c$, we  sum six generating functions similar to  equation (\ref{eq24}) and obtain the generating function $SC_a(x,y,z)$ for skew crosses of type $a$ which simplifies to :
\begin{equation*}
SC_a(x,y,z)=\frac{-16x^3y^3z^3((1-x+y)(1-x+z)+(1-y+x)(1-y+z)+(1-z+x)(1-z+y)
)}
{(1-x)^2(1-y)^2(1-z)^2(1-y-z)(1-x-y)(1-x-z)}
\end{equation*}

\paragraph{\bf Type $\boldsymbol b$} To establish the generating function of skew crosses of type $b$, we choose three faces of the cell $c$ incident to one vertex of $c$. We choose each cell in contact with a face of $c$ to be the corner cell of a $2D$ corner-polyomino. There are two possibilities once the three corner cells are chosen. Here is the generating function for a given set of three faces corresponding to one vertex of $c$ :
\begin{align*}
P_{c,z\geq 2}(x,z)\times P_{c,x\geq 2}(x,y)\times P_{c,y\geq 2}(y,z)+ P_{c,y\geq 2}(x,y)\times P_{c,z\geq 2}(y,z)\times P_{c,x\geq 2}(x,z)
\end{align*}

There are $8$ sets of three faces of $c$ incident to one vertex of $c$  and for each of these sets, we obtain the same generating function which means that the generating function for  skew crosses of type $b$ is the following:
\begin{align*}
SC_b(x,y,z)&=8( P_{c,z\geq 2}(x,z)\times P_{c,x\geq 2}(x,y)\times P_{c,y\geq 2}(y,z)\\
&+ P_{c,y\geq 2}(x,y)\times P_{c,z\geq 2}(y,z)\times P_{c,x\geq 2}(x,z))\\
&=
\frac{16x^3y^3z^3((1-x+y)(1-x+z)+(1-y+x)(1-y+z)+(1-z+x)(1-z+y)
-4)}
{(1-x)^2(1-y)^2(1-z)^2(1-y-z)(1-x-y)(1-x-z)}.
\end{align*}
To obtain the generating function for all skew crosses $SC(x,y,z)$, we simply add the generating functions for types $a$ and $b$ and observe that they have the same denominator and that terms in  the numerator cancel except for a constant  so that we obtain
 \begin{equation}\label{eq25}
SC(x,y,z)=\frac{64x^3y^3z^3}{(1-x-y)(1-x-z)(1-y-z)(1-x)^2(1-y)^2(1-z)^2}.
\end{equation}

\section{Main result}\label{sec4}
So far we have established the generating function for three disjoint classes of $3D$ polyominoes. We claim that the union of these three classes forms the whole set of $3D$ inscribed polyominoes with minimal volume. 
\begin{theorem}\label{th3}
The total number $p_{3D,min}(b,k,h)$ of polyominoes inscribed in a $b\times k\times h$ rectangular prism and minimal volume $b+k+h-2$ is the sum of diagonal polyominoes and non diagonal polyominoes of type $2D\times 2D$ and skew crosses:
\begin{equation*}
p_{3D,min}(b,k,h)=diag(b,k,h)+p_{2D\times 2D}(b,k,h)+sc(b,k,h).
\end{equation*}
\end{theorem} 
\begin{proof}
In order to prove this result, we introduce a second classification of $3D$ polyominoes and we show that every set of polyominoes forming this classification belongs to one of our three families of polyominoes. 

Consider the orthogonal projection $\Pi(P)$ of an inscribed $3D$ polyomino $P$ on the upper face of the prism. Observe that $\Pi(P)$ is a $2D$ inscribed polyomino of minimal area and therefore possesses the geometric structure {\it hook-stair-hook} of minimal $2D$ polyominoes. Two cells of the $3D$ polyomino will play a special role in our classification.  We call them {\it contact cells} and define them as follow. For every polyomino $P\in P_{3D,min}(b,k,h)$ there is a unique $3D$ stair connecting the lower and upper faces of the prism which forms a non decreasing path from floor to ceiling. The two contact cells $c_1, c_2$ are respectively, the last cell touching the floor and the first cell touching the ceiling in this path. 
We will use the positions of the projections  $\Pi(c_1),\Pi(c_2)$  in our classification. If, without loss of generality, we fix a $2D$ diagonal in the upper face  to give a direction to the {\it hook-stair-hook} structure, there are ten positions of the pair $\Pi(c_1),\Pi(c_2)$ with respect to the upper hook, each pair giving a class in this classification of  $P_{3D,min}(b,k,h)$. The ten positions can be seen in figure \ref{fig6} where $\Pi(c_1)$ and $\Pi(c_2)$ are in black. Observe that these ten cases do not form a complete partition of the set  $P\in P_{3D,min}(b,k,h)$ but our goal is to provide a complete set of representatives up to symmetry so that every other case is similar to one of the cases considered. 

Another geometric observation used throughout this proof is the fact that minimality forces  the cells that are not part of the $3D$ stair from $c_1$ to $c_2$ to form connected components that are all horizontal and attached to the $3D$ stair. Therefore the only cells that do not appear in the projection $\Pi(P)$ belong to the $3D$ stair between the two contact cells. We will show that the polyominoes belonging to each of the $10$ cases also belong to one of the three families of polyominoes, namely diagonal, $2D\times 2D$ and skew crosses.\\ 

\begin{figure}[htbp]\label{fig6}
\begin{center}
\includegraphics[width=.90\textwidth]{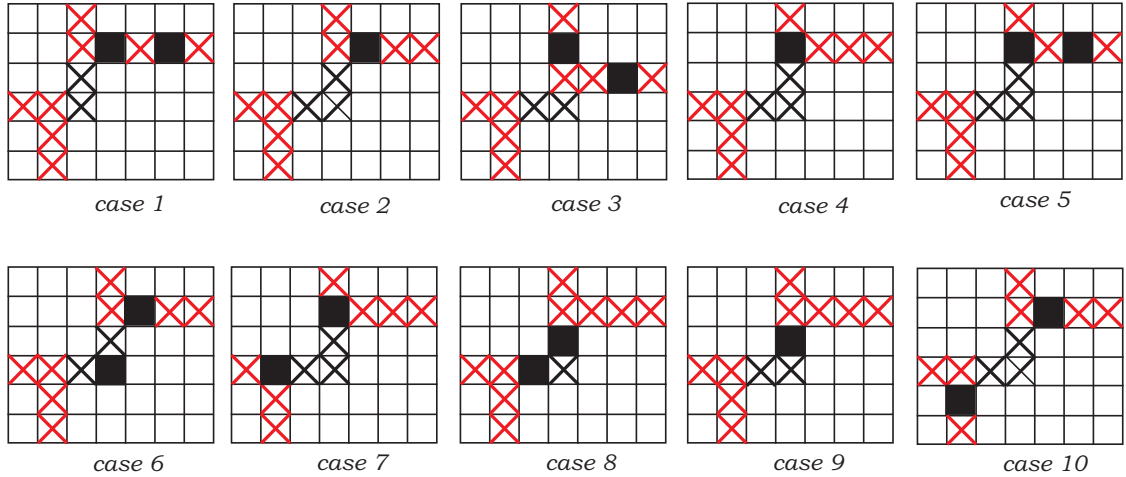}
\caption{ Classifcation of the projection of  $3D$  polyominoes on the upper face of the prism}
\end{center}
\end{figure}

\noindent
\paragraph{\bf Cases 1 and 2}
The polyominoes of cases $1$ and $2$ are $2D\times 2D$ polyominoes and are illustrated in figure \ref{fig:4:a} and \ref{fig:4:b}. Indeed the fact that the contact cells are projected on the same arm of a  $2D$ hook implies that the $3D$ stair is in a vertical plane. The other cells of the polyomino form two connected components. One component is a pilar that ends the hook arm containing $\Pi(c_1),\Pi(c_2)$ and the other contains the other hook arm which is a pilar and a $2D$ corner. \\

\noindent
\paragraph{\bf Case 3}
In case $3$, there is a contact cell projected on each arm of a $2D$ hook (see figure \ref{fig:4:c}). The part of the polyomino that is not projected on this hook is a $2D$ horizontal corner-polyomino that is neither on the lower face nor on the upper face of the prism. The two parts projected on each arm of the hook form also a $2D$ corner-polyomino so that we have a total of three $2D$ corner-polyominoes mutually perpendicular all connected to a central cell projected on the corner of the hook thus forming a skew cross.\\

\noindent
\paragraph{\bf Case 4} In case $4$, we have a full vertical pilar from floor to ceiling projected on the corner of a $2D$ hook.
Thus each part of the polyomino projected on an arm of the hook must be a horizontal pilar and so the remaining part of the $3D$ polyomino must be a  $2D$ horizontal corner-polyomino.  We thus have three horizontal connected components at various heigths with two possible scenarios. Either the $2D$ corner-polyomino is between the two pilars (figure  \ref{fig:7:a}) or it is not (figure  \ref{fig:7:b}). In the first case, we have a skew cross. In the second case we have a $3D$ diagonal polyomino. \\

\noindent
\paragraph{\bf Case 5} In case $5$ (see figure  \ref{fig:7:c}), there is a $3D$ stair in a vertical plane which is not a pilar and is projected on the arm of a $2D$ hook, two horizontal pilars whose projection completes the arms of the hook. The remaining component  must be a $2D$ horizontal corner-polyomino. There are again two subcases. Either the $2D$ corner-polyomino is at a height between the height of the two horizontal pilars either it is not.  In the first case, the polyomino is a skew cross (figure  \ref{fig:7:c}). In the other case, the polyomino is a $3D$ diagonal. \\

\begin{figure}
 \centering
\subfigure[] 
{ \label{fig:7:a}
    \includegraphics[width=3.6 cm]{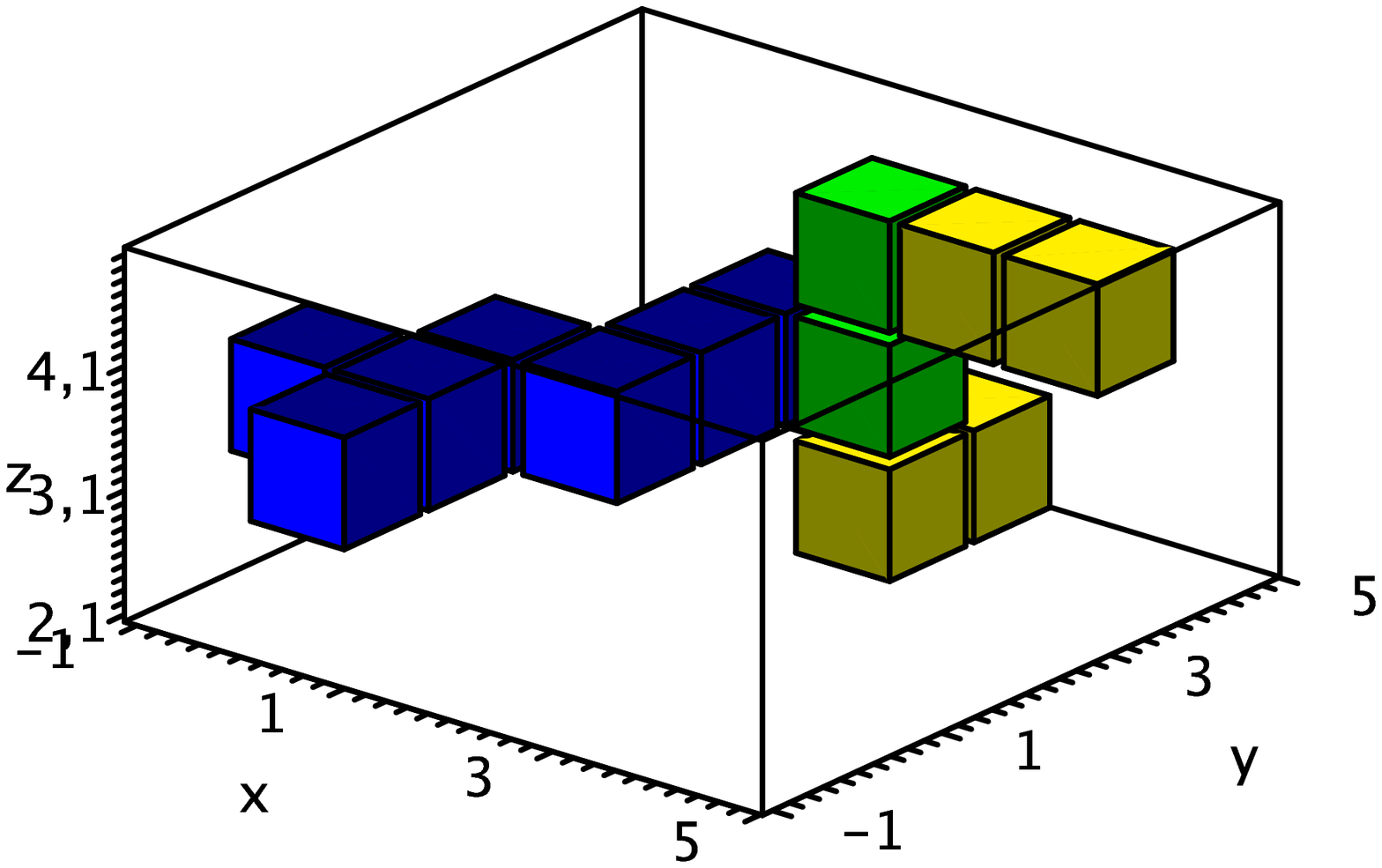}}
\hspace{.0 cm}
\subfigure[] 
{\label{fig:7:b}
    \includegraphics[width=3.6 cm]{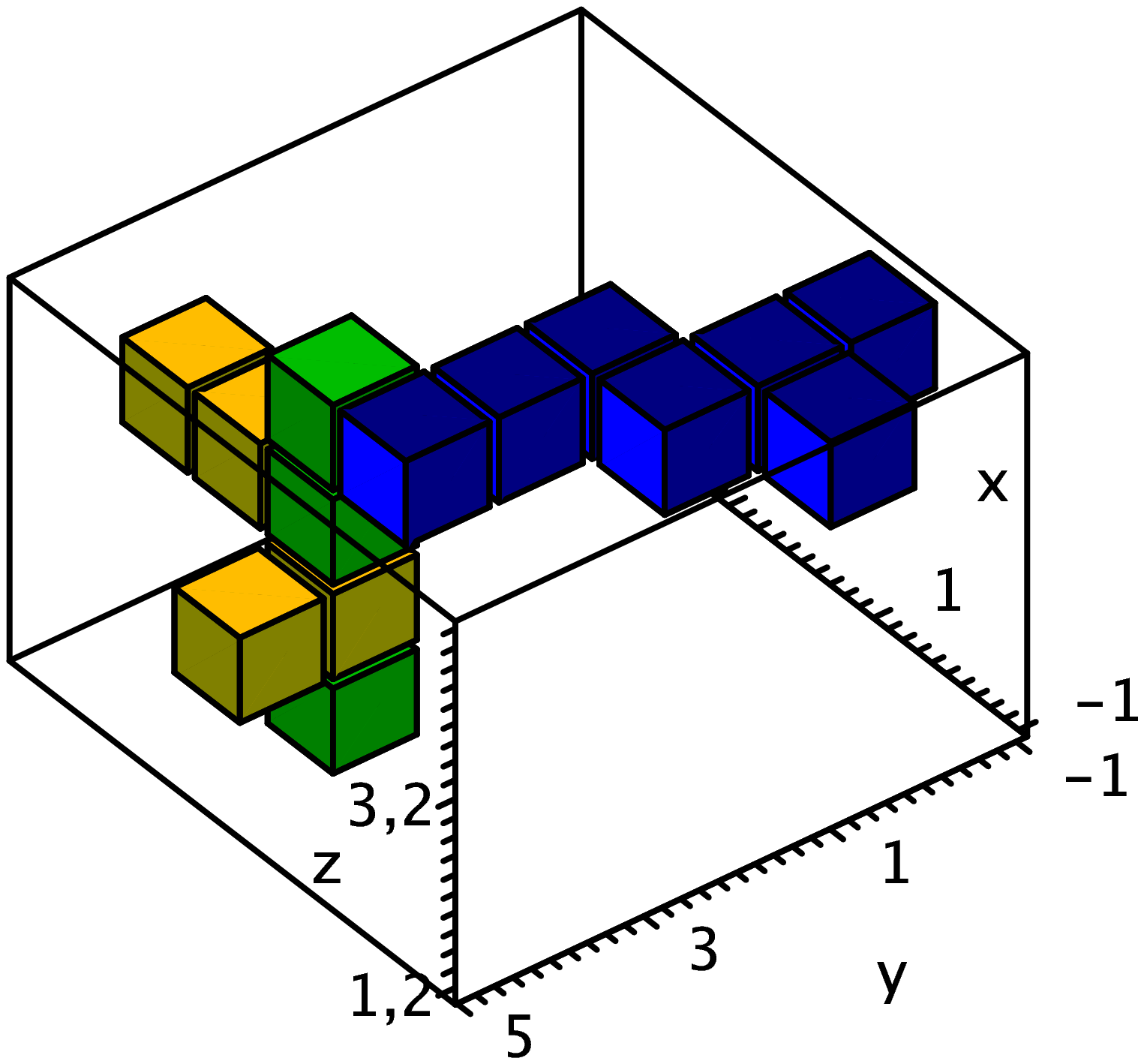}}
    \hspace{.0 cm}
     \subfigure[] 
{ \label{fig:7:c}
    \includegraphics[width=3.6 cm]{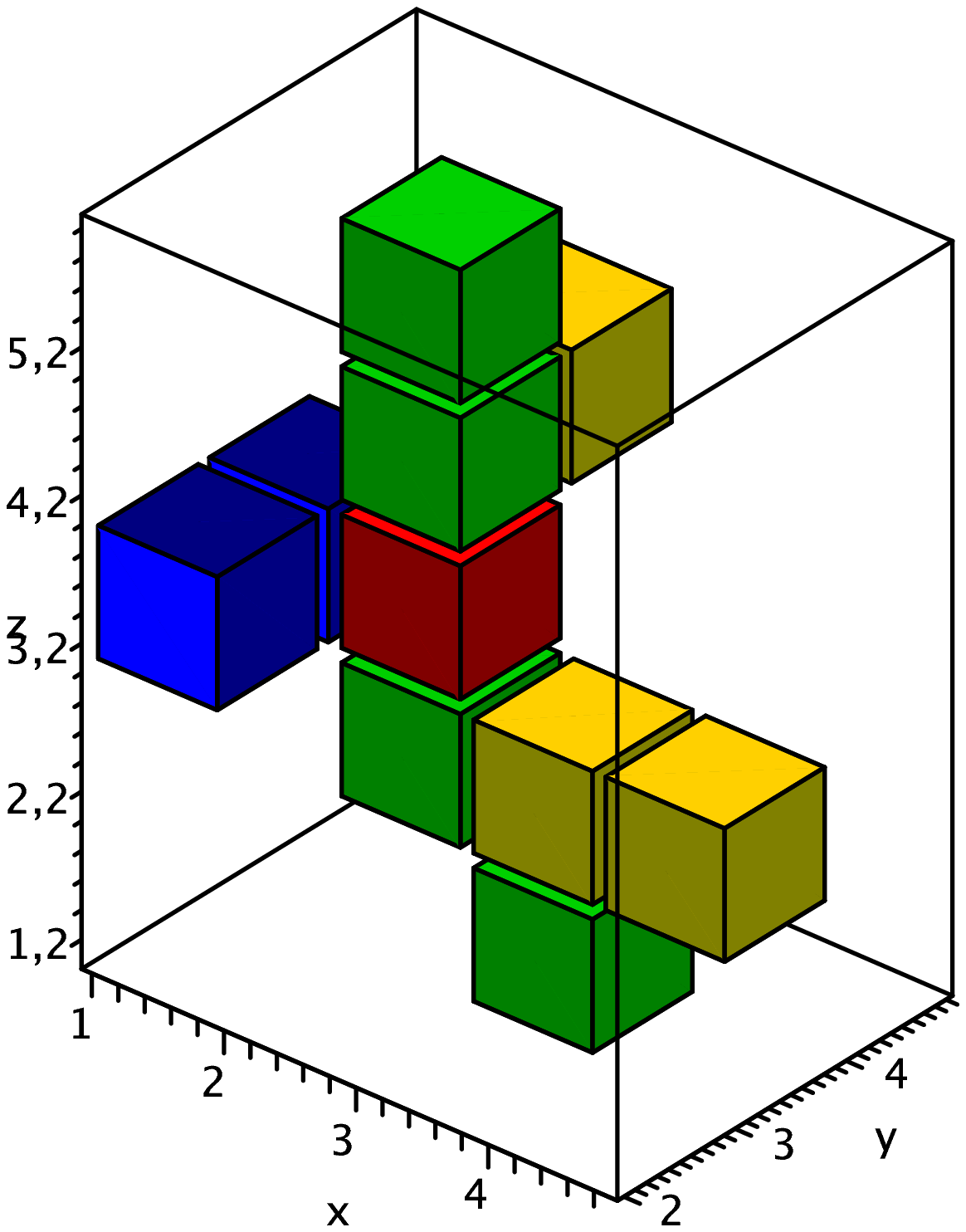}}
\hspace{.0 cm}
    \subfigure[] 
{ \label{fig:7:d}
    \includegraphics[width=3.6 cm]{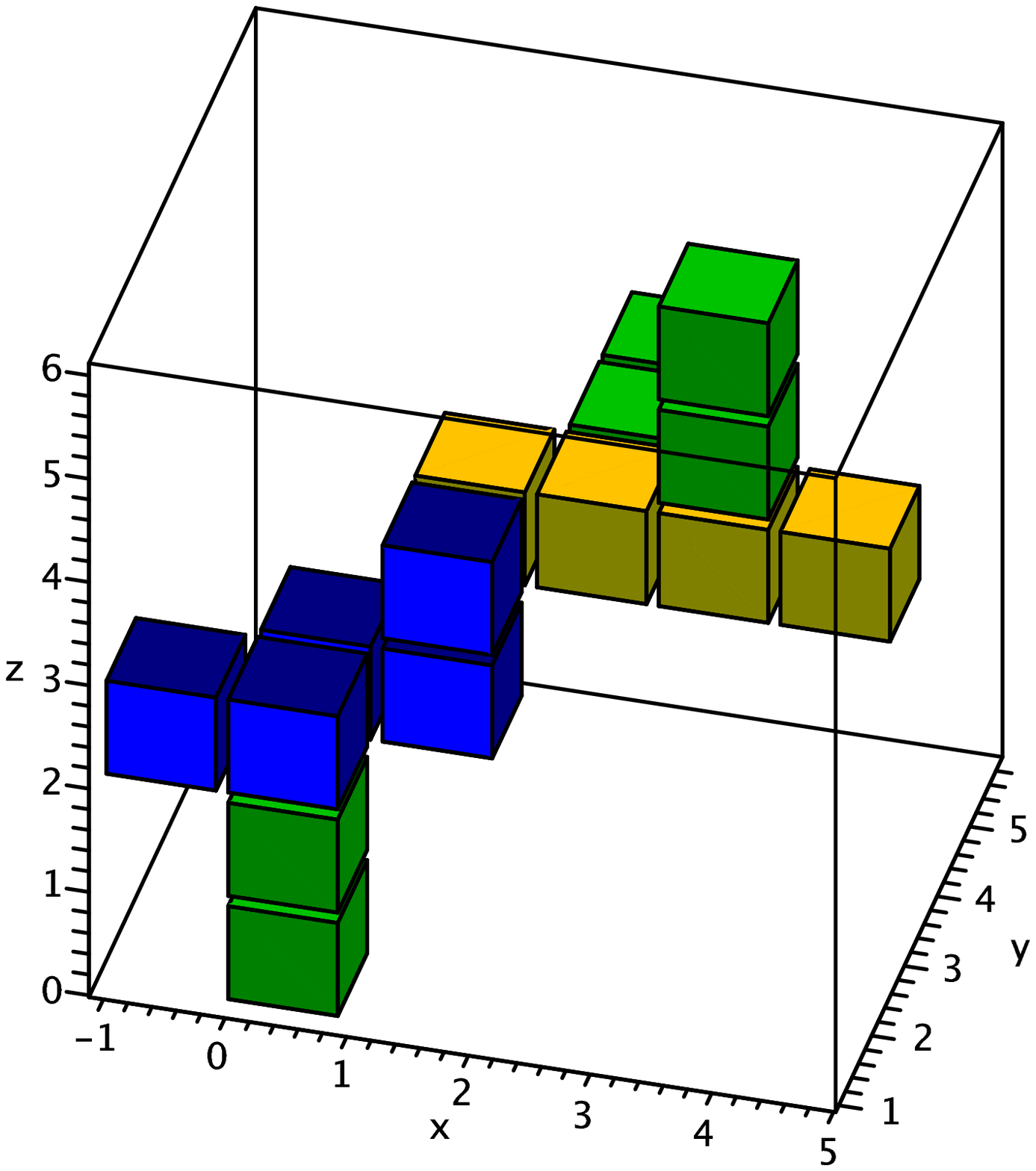}}
\hspace{.0 cm}
 \caption{Polyominos  of cases $4$, $5$ and $6$} \label{fig7}
  \end{figure}

\noindent
\paragraph{\bf Case 6}Here we have a $3D$ stair between the two contact cells whose projection joins a cell inside a hook with a cell outside the hook. Again, since the other connected components are horizontal $2D$ polyominoes, there are two horizontal pilars whose projection complete the arms of the hook. These constraints are sufficient to imply that the polyomino is a $3D$ diagonal. \\
\noindent
\paragraph{\bf Case 7} Here the $3D$ stair connecting the two contact cells is projected on the $2D$ stair so that the remaining cells form  $4$ horizontal pilars, or $3$ pilars when there is no second hook on $\Pi(P)$ (figure \ref{fig:8:a}). No matter what the relative heigth of these pilars is, they always form a $3D$ diagonal polyomino. \\

\noindent
\paragraph{\bf Cases $8$ and $9$} In these cases, the polyomino can always be decomposed in three parts: one $3D$ stair connecting the contact cells and two $2D$ horizontal corner-polyominoes with corner in contact with the $3D$ stair (figures \ref{fig:8:b} and \ref{fig:8:c})). This always give a $3D$ diagonal polyomino.  \\

\noindent
\paragraph{\bf Case $10$} Here the contact cells are projected on different $2D$ hooks.  These polyominoes always contain a $3D$ stair between  the two cells projected on the corners of the $2D$ hooks. Each of these two cells is the corner cell of  a $3D$ hook.  Thus the three parts form a $3D$ diagonal (see figure \ref{fig:8:d}) polyomino.
 \begin{figure}
 \centering
\subfigure[] 
{ \label{fig:8:a}
    \includegraphics[width=3.6 cm]{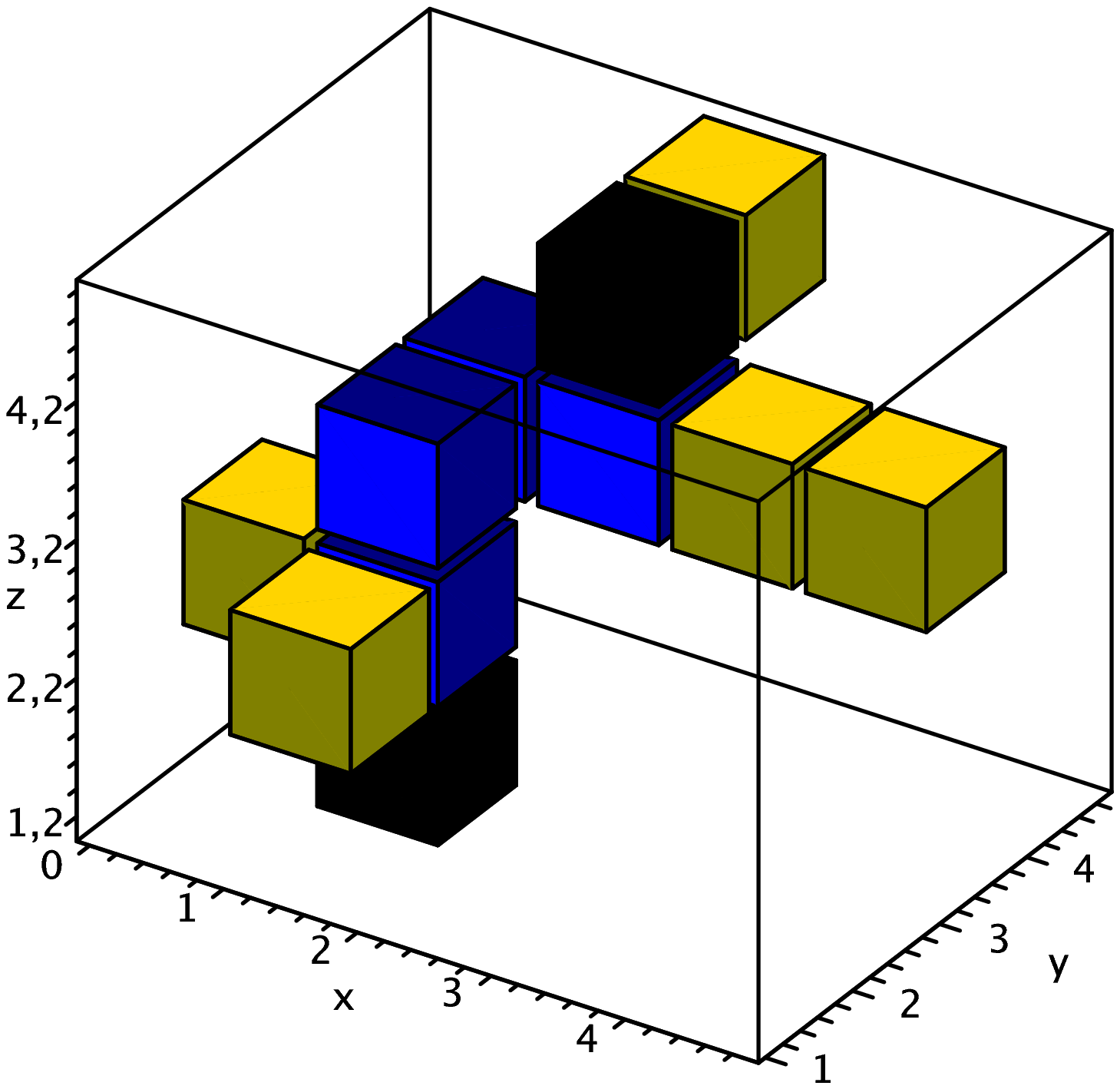}}
\hspace{.0 cm}
\subfigure[] 
{\label{fig:8:b}
    \includegraphics[width=4 cm]{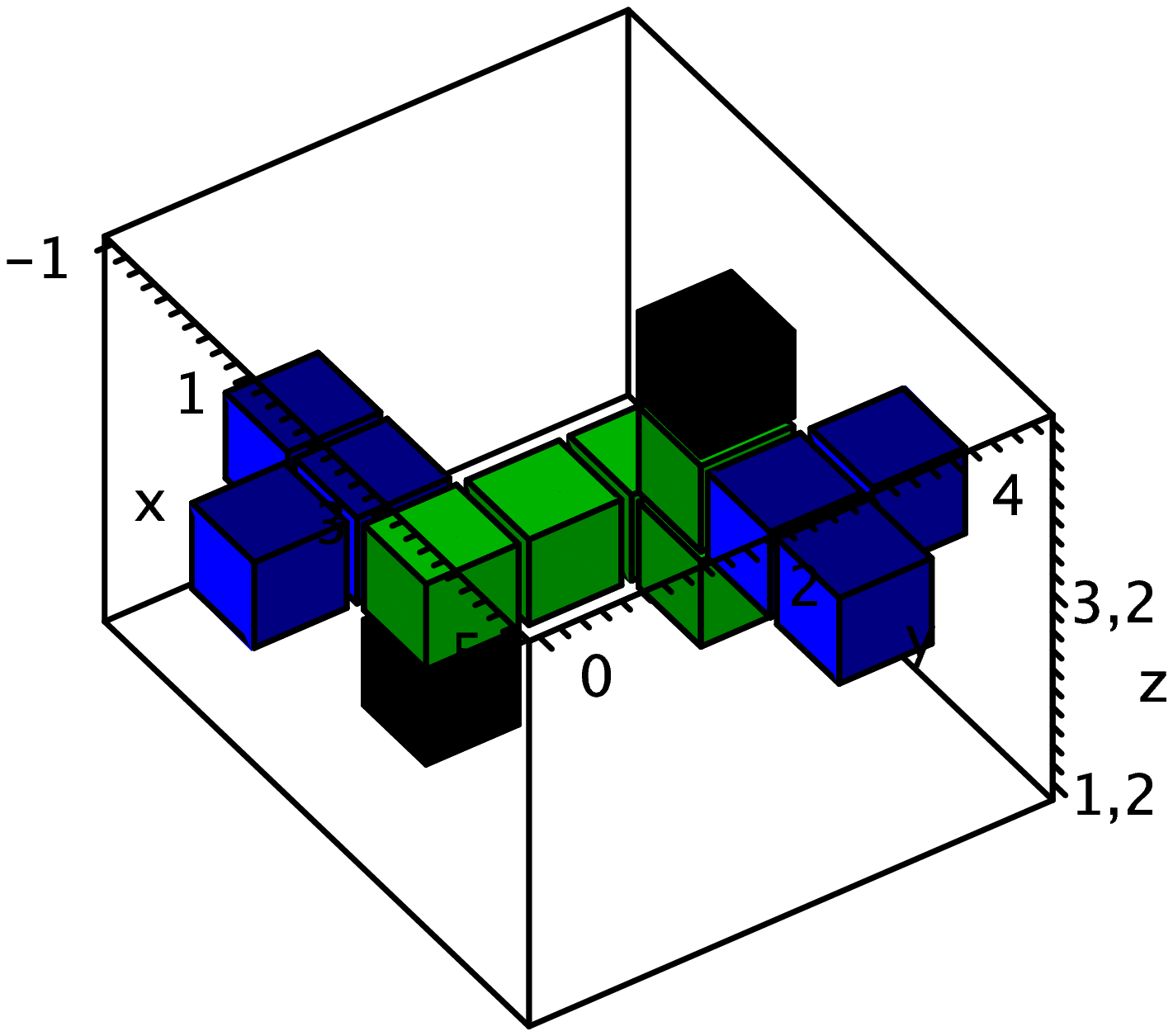}}
    \hspace{.0 cm}
\subfigure[] 
{\label{fig:8:c}
    \includegraphics[width=3.6 cm]{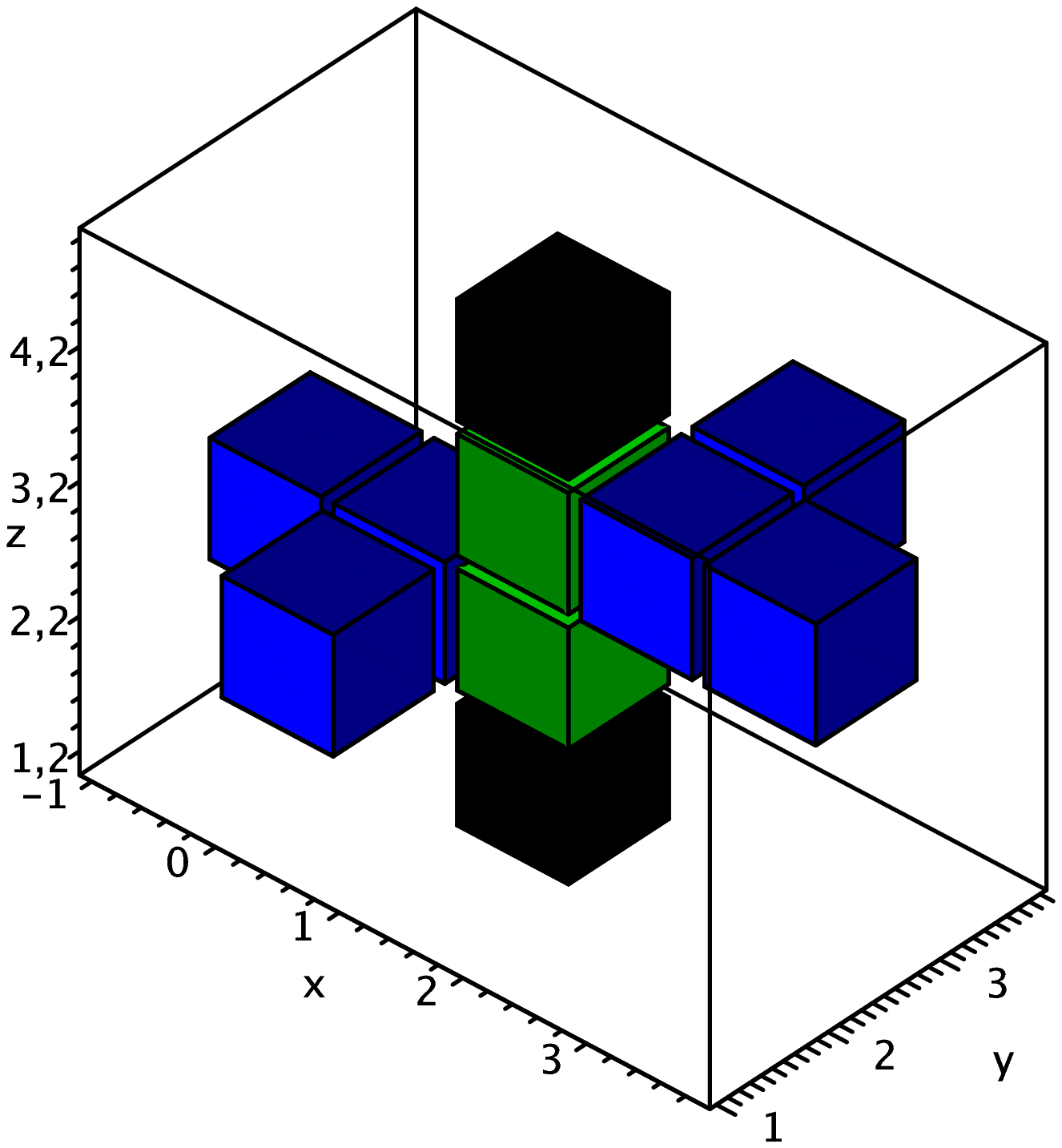}}
    \subfigure[] 
{\label{fig:8:d}
    \includegraphics[width=3.6 cm]{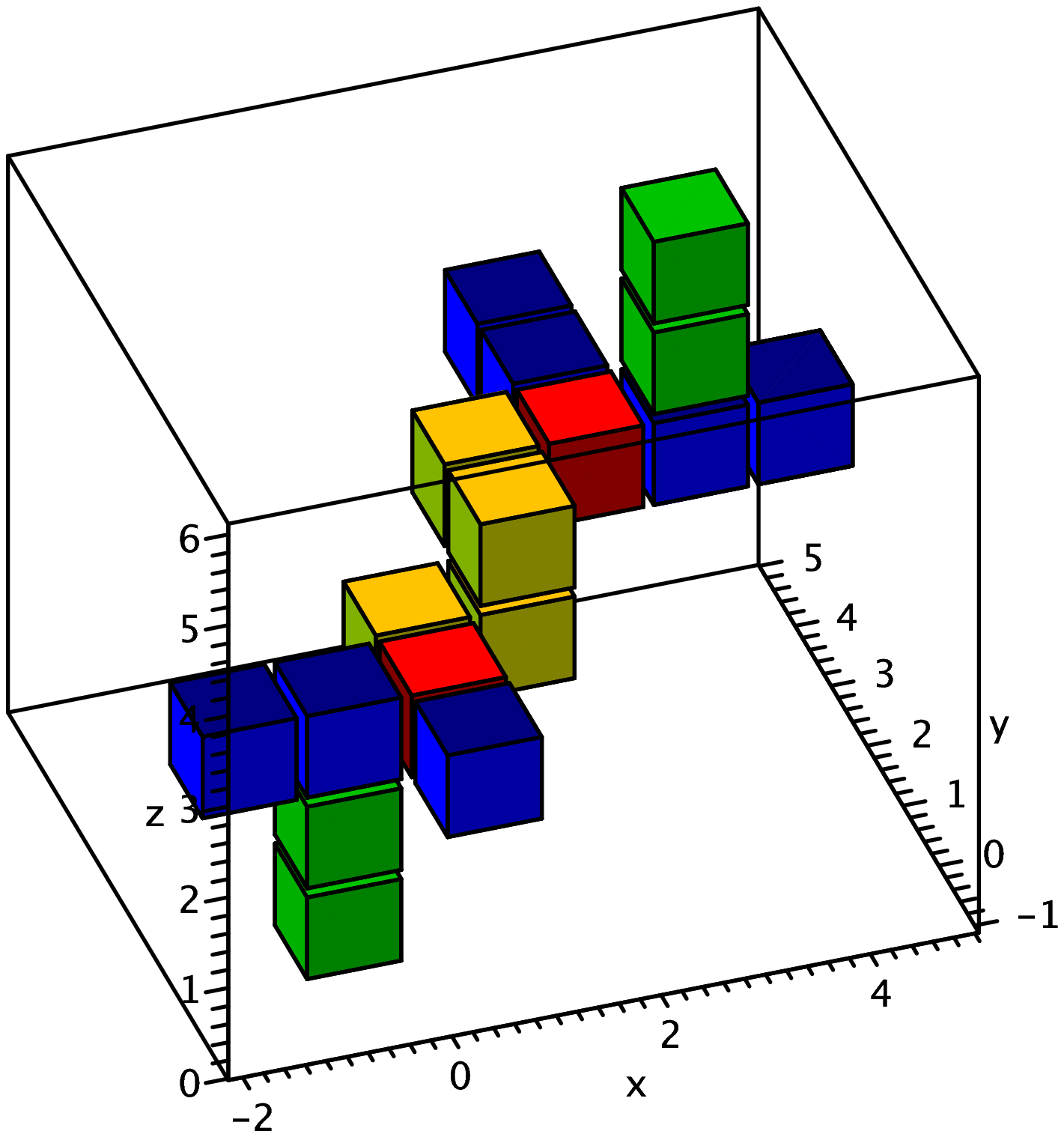}}
 \caption{Polyominos  of cases $7$, $8$, $9$, $10$} \label{fig8}
  \end{figure}
\end{proof}
\section{Exact formulas}
In  theorem \ref{th3} and other results of this paper, we chose to break generating functions in several parts because their full explicit formulation is too long. We did not provide exact formulas corresponding to all generating functions for similar reasons: the exact expressions are not always reducible. For example, here is  an exact expression for the number $sc(b,k,h)$ of skew crosses inscribed in a $b\times k\times h$ prism true for  integers $b\geq3,\; k\geq 3,\; h\geq 3$ that we could not reduce.

\begin{equation}
sc(b,k,h)=64\sum_{i=0}^{b+k-6}\sum_{r=0}^{i}\sum_{j=0}^{b-3-r}
\binom{b}{3+r+j}\binom{k}{3+i+j-r}\binom{h}{3+i}
\end{equation}
but if we turn our interest to the number of all minimum inscribed polyominoes  of a given volume $n$, we obtain interesting exact formulas that lead to asymptotic information. In what follows, we will give exact formulas for each of the three families of $3D$ polyominoes  presented in  sections $2$ and $3$ to obtain one for the set $P_{3D,min}(n)$ of inscribed minimal polyominoes of volume $n$.\\

\paragraph{\bf Skew crosses}The recipe is the same for the three families : setting $x=y=z$ in the generating function (\ref{eq25}), we obtain the generating function
 \begin{equation*}
SC(x)=\sum_{n\geq 1}sc(n)x^{n+2}=
\frac{64x^9}{(1-2x)^3(1-x)^6}
\end{equation*}
and the exact formula for the number $sc(n)$ of skew crosses of volume $n$:
 \begin{equation}\label{eq27}
sc(n-2)=2^{n+2}\left(n^2-27n+194\right)-8\left(\frac{n^5}{15}+
\frac{11n^3}{3}+12n^2+\frac{844n}{15}+96\right)
\end{equation}
\paragraph{\bf $\mathbf{2D\times2D}$ polyominoes} We set again $x=y=z$ in equations (\ref{eq18}) and  (\ref{eq19}) to obtain
\begin{align}\label{eq28}
\nonumber
P_{2D\times2D}(x)=&\sum_{n\geq 1}p_{2D\times2D}(n)x^{n+2}=
\frac{6x^8(1+2x)^2}{(1-2x)^2(1-x)^7}\\
p_{2D\times2D}(n-2)=&3\cdot2^{n+2}\left(n-15 \right)+
\left( \frac{3}{40}n^6-\frac{33}{40}n^5+\frac{65}{8}n^4
-\frac{183}{8}n^3+\frac{544}{5}n^2+\frac{147}{10}n+234\right)
\end{align}

\paragraph{\bf Diagonal polyominoes}For diagonal plyominoes, we have to modify the generating function $Diag(x,y,z)$ so that it becomes exact also for terms containing one of the variables $x,y,z$ with  power $1$. This modification is done by removing from  $1Diag(x,y,z)$ the degenerate cases counting $2D$ polyominoes which cannot be part of the inclusion-exclusion calculus in equation \ref{eq15}. Then putting $x=y=z$ we get
\begin{align}\label{eq29}
\nonumber
Diag(x)=&\sum_{n\geq 3}diag(n)x^{n+2}=
\frac{x^3\left( 36x^8+129x^6-207x^5+234x^4-126x^3+49x^2-10x+1\right)}{(1-3x)(1-2x)^2(1-x)^6}\\
diag(n-2)=&\frac{121}{48}3^n-2^n(45n-411) -\left(
\frac{53}{120}n^5-\frac{15}{8}n^4+\frac{823}{24}n^3-6n^2+\frac{22711}{60}n+\frac{4995}{16}
\right)
\end{align}
so that, adding equations  (\ref{eq27}),  (\ref{eq28}), (\ref{eq29}),we finally obtain an exact formula for $p_{3D,min}(n)$.
\begin{proposition}The generating function and exact formula for the numbers $p_{3D,min}(n)$ of $3D$ inscribed minimal polyominoes of volume $n$ are 
\begin{align*}
P_{3D,min}(n)&=\sum_{n}p_{3D,min}(n)x^{n+2}\\
&=\frac{x^3(72x^{10}+36x^9+510x^8-1117x^7+1276x^6-1155x^5+710x^4-293x^3+81x^2-13x+1)}{(1-3x)(1-2x)^3(1-x)^7}\\
p_{3D,min}(n)&=\frac{11^2\cdot 3^{n+1}}{16}+2^{n+2}(4n^2 - 125n + 741)\\
&\hspace{.5cm}+ \frac{3n^6}{40} - \frac{9n^5}{10} - \frac{7n^4}{2} - \frac{133n^3}{2} - \frac{1931n^2}{5} - \frac{31727n}{20} - \frac{47739}{16}
\end{align*}
\end{proposition}
\begin{table}[htdp]
\begin{center}
\begin{tabular}{|c|c|c|c|c|c|c|c|c|c|c|c|c|c|c|}
\hline
$n$&1&2&3&4&5&6&7&8&9&10\\
\hline
$p_{3D,min}(n)$&1&3& 15& 83& 450& 2295&10834&47175&190407&719243\\
\hline
\end{tabular}
\end{center}
\caption{Numbers $p_{3D,min}(n)$ of minimal polyominoes of volume $n$} 
\label{tab2}
\end{table}
\paragraph{\bf An exact formula for $P_{3D,min}(2,b,k)$} It is possible and relatively easy  to derive one variable formulas for the numbers $P_{3D,min}(2,b,k)$ for different values of $b$ and $k$ with Maple and {\it gfun} but in the next proposition we develop combinatorially a two variables expression for $P_{3D,min}(2,b,k)$.
\begin{proposition}\label{prop6}
The numbers $P_{3D,min}(2,b,k)$ satisfy the expression
\begin{align}\label{eq30}
P_{3D,min}(2,b,k)=&
\left[16\binom{b+k-2}{b-1} - 4(b+k)\right](2b+2k-3)+4(b-2)(k-2)\\
\nonumber 
&+\left[16(b+k-2)-12bk\right](b+k-1)
\end{align}
\end{proposition}
\begin{proof}
We use the fact that a $3D$ minimal polyomino inscribed in a $2\times b\times k$ prism is obtained by adding one cell to a $2D$ minimal polyomino inscribed in a $b\times k$ rectangle so  that these $3D$ polyominos are obtained from $2D$ polyominos rooted on one cell when we set the weight of the rooted cell $x$ to be  $2^{deg(x)}$. We have
\begin{align}\label{eq31}
p_{3Dmin}(2,b,k)=\sum_{p\in P_{2Dmin}(2,b,k)}\sum_{x\in p}2^{deg(x)}
\end{align}
Most cells in a minimal $2D$ polyomino have degree $2$ except the leaves of degree $1$ and one or two cells of degree $3$ or $4$. In order to transform equation \ref{eq31} into equation \ref{eq30}, we have to partition  the set $P_{2D,min}(b,k)$ with respect to the number of leaves and the number of cells of degree $3$ and $4$ . 

\paragraph{  Two leaves}  $2D$ Polyominoes with two leaves are the $2\binom{b+k-2}{b-1}$ stairs polyominoes inscribed in a $b\times k$ rectangle. In such a polyomino, there are $b+k-3$ cells of degree two and two cells of degree one.  The number of $3D$ polyominoes obtained from these polyominoes is
\begin{align}\label{eq32}
2\binom{b+k-2}{b-1}\left[ (b+k-3)\times 2^2+2\times 2^1\right].
\end{align}\\
\paragraph{Three leaves} Polyominos with three leaves are the $2D$ corner-polyominos minus the hooks and the stairs and there are $4\binom{b+k-2}{b-1}-2(b+k)$ of them in a $b\times k$ rectangle. Each of these polyomino possesses $b+k-5$ cells of degree two, three cells of degree one and one cell of degree three. The number of $3D$ polyominoes obtained from these polyominoes is thus
\begin{align}\label{eq33}
\left[4\binom{b+k-2}{b-1}-2(b+k)\right]\left[ (b+k-5)\times 2^2+3\times 2^1
+2^3\right].
\end{align}\\
\paragraph{ Four leaves} Polyominos with four leaves are either crosses or non degenerate {\it hook-stair-hook} structures. There are $(b-2)(k-2)$ crosses with four leaves and $2\binom{b+k-2}{b-1}+4(b+k-2)-3bk-(b-2)(k-2)$ 
non degenerate {\it hook-stair-hooks} in a $b\times k$ rectangle. In a cross there are $b+k-6$ cells of degree two, four cells of degree one and one cell of degree four. In a non degenerate {\it hook-stair-hook}, there are  $b+k-7$ cells of degree two, four cells of degre one and two cells of degree three. This gives the following number of $3D$ polyominoes :
\begin{align}\label{eq34}
&(b-2)(k-2)\left[ (b+k-6)\times 2^2+4\times 2^1
+2^4\right]+\\
\nonumber
&\left[2\binom{b+k-2}{b-1}+4(b+k-2)-3bk-(b-2)(k-2)\right]
\left[ (b+k-7)\times 2^2+4\times 2^1+2\times 2^3\right].
\end{align}
The sum of expressions (\ref{eq32}),(\ref{eq33}) and (\ref{eq34}) gives 
equation (\ref{eq30}).
\end{proof}
We can easily extend the argument in the above proof  to obtain exact expressions for $P_{3D,min}(3,b,k)$ from two-rooted $2D$ polyominoes and so on for $P_{3D,min}(4,b,k)$, etc. For $P_{3D,min}(3,b,k)$ with $b\geq2, k\geq 2$, we obtain 
\begin{align*}
P_{3D,min}(3,b,k)=&8(b^3+k^3)-12(b^3k+bk^3)-24b^2k^2-46(b^2+k^2)+41(b^2k+bk^2) -93kb +58(b+k)-8\\
& +4\binom{b+k-2}{b-1}(4b+4k-1)(2b+2k-3).
\end{align*}
\paragraph{\bf Remarks}
\begin{enumerate}
\item  In parallel with this work, one of the authors (H. Cloutier), wrote two programs  to count minimal inscribed polyominoes. One program uses formulas obtained from the projection $\Pi(P)$ of the polyomino on the ceiling of the prism. The other program runs through all $3D$ polyominoes and keeps only the needed ones. We used the datas obtained from these programs to validate the results of this paper. 
\item The argument in the proof of proposition \ref{prop6} can be used to obtain exact formulas for $3D$  polyominoes inscribed in a prism with volume $min+1$ when we use the formula giving $2D$ polyominoes of area $min+1$ in (\cite{GCN}). 
\item In the introduction of this paper, we asked a question about the relation between $2D$ and $3D$ combinatorics. In this work, we had surprises moving from $2D$ to $3D$ but the answer to the question is not clear to us yet and we postpone our judgement. 
\item The diagonal subseries $P_{3D,min}(t)=\sum_np_{3D,min}(n,n,n)t^n$ obtained from $P_{3D,min}(x,y,z)$ by setting all equals the exponents of $x,y,z$  satisfies a functional equation of degree six in $P_{3D,min}(t)$ with coefficients that are polynoms in $t$. But no exact expression for $p_{3D,min}(n,n,n)$ could be found.
\end{enumerate}
\newpage

\end{document}